\documentclass[12pt,reqno]{amsart}
\usepackage{amsmath , amssymb }
\usepackage{graphics}

\usepackage[mathscr]{eucal}

\newdimen\unit\newdimen\psep\newcount\nd\newcount\ndx\newbox\dotb\newbox\ptbox
\newdimen\dx\newdimen\dy\newdimen\dxx\newdimen\dyy\newdimen\hgt
\newdimen\xoff\newdimen\yoff
\newcommand\clap[1]{\hbox to 0pt{\hss{#1}\hss}}
\newcommand\vdisk[1]{{\font\dotf=cmr10 scaled #1\dotf.}}
\newcommand\varline[2]{\setbox\dotb\hbox{\vdisk{#1}}\xoff=-.5\wd\dotb
\wd\dotb=0pt\yoff=-.5\ht\dotb\psep=#2\ht\dotb}
\newcommand\varpt[1]{\setbox\ptbox\clap{\vdisk{#1}}\setbox\ptbox
\hbox{\raise-.5\ht\ptbox\box\ptbox}}
\newcommand\cpt{\copy\ptbox}
\newcommand\point[3]{\rlap{\kern#1\unit\raise#2\unit\hbox{#3}}}
\newcommand\setnd[4]{\dx=#3\unit\advance\dx-#1\unit\divide\dx by\psep
\dy=#4\unit\advance\dy-#2\unit\divide\dy by\psep \multiply\dx
by\dx\multiply\dy by\dy\advance\dx\dy\nd=1\advance\dx-1sp
\loop\ifnum\dx>0\advance\dx-\nd sp\advance\nd1\advance\dx-\nd
sp\repeat}

\newcommand\dline[5]{{\nd=#5\hgt=#2\unit\dx=#3\unit\advance\dx-#1\unit
\divide\dx by\nd\dy=#4\unit\advance\dy-#2\unit\divide\dy by\nd
\advance\hgt\yoff\rlap{\kern#1\unit\kern\xoff\loop\ifnum\nd>1\advance\nd-1
\advance\hgt\dy\kern\dx\raise\hgt\copy\dotb\repeat}}}

\newcommand\qellip[4]{{\setnd{0}{0}{#3}{#4}\dx=\unit\dy=0pt\raise\yoff\rlap{%
\kern#1\unit\kern\xoff\raise#2\unit\hbox{\loop\ifnum\dx>0\rlap{\kern#3\dx
\raise#4\dy\copy\dotb}\hgt=\dx\divide\hgt
by\nd\advance\dy\hgt\hgt=\dy \divide\hgt
by\nd\advance\dx-\hgt\repeat\rlap{\raise#4\dy\copy\dotb}}}}}
\newcommand\bez[6]{{\setnd{#1}{#2}{#3}{#4}\ndx=\nd\setnd{#3}{#4}{#5}{#6}
\ifnum\ndx>\nd\nd=\ndx\fi\dx=#3\unit\advance\dx-#1\unit\dy=#4\unit
\advance\dy-#2\unit\dxx=#5\unit\advance\dxx-#1\unit\dyy=#6\unit\advance
\dyy-#2\unit\advance\dxx-2\dx\advance\dyy-2\dy\divide\dxx
by\nd\divide\dyy
by\nd\advance\dx.25\dxx\advance\dy.25\dyy\divide\dx
by\nd\divide\dy by\nd \multiply\nd
by2\dx=100\dx\dy=100\dy\dxx=100\dxx\dyy=100\dyy\divide\dxx by\nd
\divide\dyy
by\nd\hgt=#2\unit\raise\yoff\rlap{\kern#1\unit\kern\xoff
\raise\hgt\copy\dotb\loop\ifnum\nd>0\advance\nd-1\advance\hgt0.01\dy
\kern0.01\dx\raise\hgt\copy\dotb\advance\dx\dxx\advance\dy\dyy\repeat}}}
\newcommand\ptu[3]{\point{#1}{#2}{\cpt\raise1ex\clap{$\scriptstyle{#3}$}}}
\newcommand\ptd[3]{\point{#1}{#2}{\cpt\raise-1.8ex\clap{$\scriptstyle{#3}$}}}
\newcommand\ptr[3]{\point{#1}{#2}{\cpt\raise-.4ex\rlap{$\ \scriptstyle{#3}$}}}
\newcommand\ptl[3]{\point{#1}{#2}{\cpt\raise-.4ex\llap{$\scriptstyle{#3}\ $}}}
\newcommand\ptlu[3]{\point{#1}{#2}{\raise.8ex\clap{$\scriptstyle{#3}$}}}
\newcommand\ptld[3]{\point{#1}{#2}{\raise-1.6ex\clap{$\scriptstyle{#3}$}}}
\newcommand\ptlr[3]{\point{#1}{#2}{\raise-.4ex\rlap{$\,\scriptstyle{#3}$}}}
\newcommand\ptll[3]{\point{#1}{#2}{\raise-.4ex\llap{$\scriptstyle{#3}\,$}}}

\newcommand\thnline{\varline{400}{.6}}

\varpt{2500}\thnline\unit=2em

\newtheorem{thm}{Theorem}

\newtheorem*{vBKlemma}{The van den Berg--Kesten Lemma}
\newtheorem{conj}{Conjecture}

\newtheorem{prob}{Problem}
\newtheorem{lemma}[thm]{Lemma}
\newtheorem{prop}[thm]{Proposition}

\newtheorem{cor}[thm]{Corollary}

\newtheorem{obs}[thm]{Observation}
\theoremstyle{definition}\newtheorem{rmk}{Remark}
\theoremstyle{definition}\newtheorem*{defn}{Definition}
\newcommand{\ds}{\displaystyle}
\newcommand{\ul}{\underline}

\def\A{\mathcal{A}}

\def\C{\mathcal{C}}

\def\E{\mathcal{E}}

\def\J{\mathcal{J}}

\def\Q{\mathcal{Q}}

\def\S{\mathcal{S}}

\def\U{\mathcal{U}}
\def\W{\mathcal{W}}
\def\X{\mathcal{X}}

\def\Ex{\mathbb{E}}

\def\N{\mathbb{N}}
\def\Pr{\mathbb{P}}

\def\RR{\mathbb{R}}
\def\ZZ{\mathbb{Z}}

\def\le{\leqslant}
\def\ge{\geqslant}

\def\eps{\varepsilon}
\def\dim{\textup{dim}}
\def\Var{\textup{Var}}

\def\Bin{\textup{Bin}}

\def\<{\langle}
\def\>{\rangle}

\begin{document}
\title{Bootstrap percolation in high dimensions}

\author{J\'ozsef Balogh}
\address{Department of Mathematics\\ University of Illinois\\ 1409 W. Green Street\\ Urbana, IL 61801\\ and\\ Department of Mathematics\\ University of California\\ San Diego, La Jolla, CA 92093}\email{jobal@math.uiuc.edu}

\author{B\'ela Bollob\'as}
\address{Trinity College\\ Cambridge CB2 1TQ\\ England\\ and \\ Department of Mathematical Sciences\\ The University of Memphis\\ Memphis, TN 38152, USA} \email{B.Bollobas@dpmms.cam.ac.uk}

\author{Robert Morris}
\address{IMPA, Estrada Dona Castorina 110, Jardim Bot\^anico, Rio de Janeiro, RJ, Brasil} \email{rob@impa.br}
\thanks{The first author was supported during this research by NSF CAREER Grant DMS-0745185 and DMS-0600303, UIUC Campus Research Board Grants 09072 and 08086, and OTKA Grant K76099, the second by ARO grant W911NF-06-1-0076, and and NSF grants CNS-0721983, CCF-0728928 and DMS-0906634, and the third by MCT grant PCI EV-8C, ERC Advanced grant DMMCA, and a Research Fellowship from Murray Edwards College, Cambridge}

\begin{abstract}
In $r$-neighbour bootstrap percolation on a graph $G$, a set of initially infected vertices $A \subset V(G)$ is chosen independently at random, with density $p$, and new vertices are subsequently infected if they have at least $r$ infected neighbours. The set $A$ is said to percolate if eventually all vertices are infected. Our aim is to understand this process on the grid, $[n]^d$, for arbitrary functions $n = n(t)$, $d = d(t)$ and $r = r(t)$, as $t \to \infty$. The main question is to determine the critical probability $p_c([n]^d,r)$ at which percolation becomes likely, and to give bounds on the size of the critical window. In this paper we study this problem when $r = 2$, for all functions $n$ and $d$ satisfying $d \gg \log n$.

The bootstrap process has been extensively studied on $[n]^d$ when $d$ is a fixed constant and $2 \le r \le d$, and in these cases $p_c([n]^d,r)$ has recently been determined up to a factor of $1 + o(1)$ as $n \to \infty$. At the other end of the scale, Balogh and Bollob\'as determined $p_c([2]^d,2)$ up to a constant factor, and Balogh, Bollob\'as and Morris determined $p_c([n]^d,d)$ asymptotically if $d \ge (\log \log n)^{2+\eps}$, and gave much sharper bounds for the hypercube.

Here we prove the following result: let $\lambda$ be the smallest positive root of the equation
$$\ds\sum_{k=0}^\infty \frac{(-1)^k \lambda^k}{2^{k^2-k} k!} \: = \: 0,$$
so $\lambda \approx 1.166$. Then
$$ \ds\frac{16\lambda}{d^2} \left( 1 + \frac{\log d}{\sqrt{d}} \right)\: 2^{-2\sqrt{d}} \; \le \; p_c([2]^d,2) \; \le \;   \ds\frac{16\lambda}{d^2} \left( 1 + \frac{5(\log d)^2}{\sqrt{d}} \right) \: 2^{-2\sqrt{d}}$$
if $d$ is sufficiently large, and moreover
$$p_c\big( [n]^d,2 \big) \; = \; \Big( 4\lambda + o(1) \Big) \left( \frac{n}{n-1} \right)^2 \, \ds\frac{1}{d^2} \, 2^{-2\sqrt{d \log_2 n}}$$
as $d \to \infty$, for every function $n = n(d)$ with $d \gg \log n$.
\end{abstract}

\maketitle

\section{Introduction}\label{n^dintro}

Bootstrap percolation is a cellular automaton, in which an infection spreads (on a graph $G$, with threshold $r \in \N$) according to the following deterministic rule: a vertex with at least $r$ infected neighbours becomes infected, and infected sites are infected forever. To be precise, let $A = A_0 \subset V(G)$, let
$$A_{t+1} \; := \; A_t \,\cup\, \big\{ v \in V(G) \,:\, |\Gamma(v) \cap A_t| \ge r \big\}$$
for each $t \in \N_0 = \{0,1,2,\ldots\}$, and let $[A] = \bigcup_t A_t$. We say that $A$ \emph{percolates} if $[A] = V(G)$, i.e., if every vertex is eventually infected.

The bootstrap process is closely related to the Ising Model of ferromagnetism, and was first introduced in 1979 by Chalupa, Leath and Reich~\cite{CLR} in the context of statistical mechanics. The main question is as follows: Suppose that the elements of the set $A$ are chosen independently at random with probability $p$. For which values of $p$ is percolation likely to occur? More precisely, writing $\Pr_p$ for the Bernoulli distribution with density $p$, and given a graph $G$, a threshold $r \in \N = \{1,2,\ldots\}$ and $\alpha \in [0,1]$, define
$$p_\alpha(G,r) \; := \; \inf \Big\{ p \,:\, \Pr_p\big( A \textup{ percolates in $r$-neighbour bootstrap} \big) \ge \alpha \Big\}.$$
We shall write $p_c(G,r)$ for $p_{1/2}(G,r)$, the \emph{critical probability} for percolation. Our aim is to give sharp bounds on $p_c(G,r)$, and to determine the size of the critical window, $p_{1-\eps} - p_\eps$.

This problem has been extensively studied on the grid $[n]^d$, with $d$ and $r$ fixed. In particular, Schonmann~\cite{Sch92} showed that $p_c(\ZZ^d,r) \in \{0,1\}$ for all $d$ and $r$, Aizenman and Lebowitz~\cite{AL} determined $p_c([n]^d,2)$ up to a constant factor, and Cerf and Cirillo~\cite{CC} (for $d = r = 3$) and Cerf and Manzo~\cite{CM} did the same for $p_c([n]^d,r)$ for all $2 \le r \le d$. The result of Cerf and Cirillo was a particularly important breakthrough, since the problem becomes much more difficult when $r \ge 3$; in Section~\ref{probsec} we shall see that a similar phenomenon also occurs on the hypercube. In another significant breakthrough, the first sharp threshold for the critical probability was proved by Holroyd~\cite{Hol}, who showed that
$$p_c([n]^2,2) \; = \; \frac{\pi^2}{18\log n} \,+\, o\left( \frac{1}{\log n} \right).$$
The sharp threshold for all $d \ge r \ge 2$ was recently determined by Balogh, Bollob\'as and Morris~\cite{d=r=3} (for $d = r = 3$) and Balogh, Bollob\'as, Duminil-Copin and Morris~\cite{alldr} (in general).

Having proved such sharp bounds for $d$ fixed, it is natural to ask the question: what happens when $n$, $d$ and $r$ are all allowed to tend to infinity? In general, given arbitrary functions $n = n(t)$, $d = d(t)$ and $r = r(t)$, with $n + d \to \infty$ as $t \to \infty$, we would like to determine $p_c([n]^d,r)$, and show that percolation has a sharp threshold.

A first step in answering this question was provided by Balogh and Bollob\'as~\cite{BB}, who determined the critical probability for two-neighbour bootstrap percolation on the hypercube up to a constant. They proved that
$$\frac{1}{150d^2} 2^{-2\sqrt{d}} \; \le \; p_c\big( [2]^d,2 \big) \; \le \; \frac{5000}{d^2} 2^{-2\sqrt{d}}.$$
Balogh, Bollob\'as and Morris~\cite{maj} also studied the majority (i.e., $r = \lceil d/2 \rceil$) bootstrap process on the hypercube. They gave very sharp bounds on the critical probability, and moreover used the techniques developed in their proof to show that
$$p_c([n]^d,d) \; = \; \frac{1}{2} \,+\,o(1)$$
as $d \to \infty$ if $n = n(d)$ satisfies $d \ge (\log \log n)^{2+\eps}$. (As a consequence of Schonmann's proof in~\cite{Sch92}, we have $p_c([n]^d,d) = o(1)$ if $d \le \log^* n$.) Morris~\cite{me} also used the techniques of~\cite{maj} to prove that, in zero-temperature Glauber dynamics on $\ZZ^d$, the critical threshold for fixation converges to $1/2$ as $d \to \infty$ (see also~\cite{FSS}).

In this paper we shall determine a sharp threshold for the critical probability in two-neighbour bootstrap percolation on the hypercube, and more generally on all sufficiently high-dimensional grids $[n]^d$. We shall moreover give fairly tight bounds on the second term of $p_c([2]^d,2)$. Throughout the paper, let $\lambda \approx 1.16577$ denote the smallest positive root of the equation
\begin{equation}\label{xdef}
\ds\sum_{k=0}^\infty \frac{(-1)^k \lambda^k}{2^{k^2-k} k!} \: = \: 0.
\end{equation}
The following theorem is our main result.

\begin{thm}\label{hypercube}
If $d$ is sufficiently large, then
$$ \ds\frac{16\lambda}{d^2} \left( 1 + \frac{\log d}{\sqrt{d}} \right)\: 2^{-2\sqrt{d}} \; \le \; p_c([2]^d,2) \; \le \;   \ds\frac{16\lambda}{d^2} \left( 1 + \frac{5(\log d)^2}{\sqrt{d}} \right) \: 2^{-2\sqrt{d}},$$
where $\lambda \approx 1.166$ is as defined above.
\end{thm}

We shall also determine a sharp threshold for $p_c([n]^d,2)$ for all high-dimensional grids $[n]^d$. Here, and throughout, $\log$ is to the base $2$, unless otherwise stated. We write $f \gg g$ to indicate that $g(d)/f(d) \to 0$ as $d \to \infty$.

\begin{thm}\label{n^d}
Let $n = n(d)$ be a function such that $d \gg \log n \ge 1$ as $d \to \infty$. Then
$$p_c\big( [n]^d,2 \big) \: = \: \Big( 4\lambda + o(1) \Big) \left( \frac{n}{n-1} \right)^2 \, \ds\frac{1}{d^2} \, 2^{-2\sqrt{d \log_2 n}},$$
as $d \to \infty$, where $\lambda \approx 1.166$ is as defined above.
\end{thm}

In fact (see Theorems~\ref{upper} and~\ref{lower}), if $d \ge n^{16+c}$ for some constant $c > 0$ then we shall prove bounds on the second order term of $p_c([n]^d,2)$ similar to those in Theorem~\ref{hypercube}.

\begin{rmk}\label{torus}
On the torus $\ZZ_n^d$, with $n \ge 4$, straightforward calculations suggest that the critical probability is the following, slightly different function:
$$p_c\left( \ZZ_n^d,2 \right) \; = \; \ds\frac{4\lambda + o(1)}{d^2} \, 2^{-2\sqrt{d \log n}}.$$
The upper bound follows by the method of Section~\ref{uppersec}, but the simple coupling we use to prove the lower bound (see Lemma~\ref{arbcube}) is no longer valid for the torus. The lemma should still be true for $n \ge 4$; for $n = 3$, on the other hand, the lemma is false and the critical probability is very different. See Section~\ref{probsec} for further discussion of these issues.
\end{rmk}

The hypercube is a very well-studied combinatorial object; for example, see the classical work (relating to a different percolation problem on the hypercube) of~\cite{ES,AKS,BKL}, or the more recent improvements~\cite{BCHSS,HS1,HS2}. We also note that the bootstrap process has been studied on other graphs, such as infinite trees~\cite{CLR,FS,BPP}, the random regular graph~\cite{BPi,Svante}, and a more general class of `locally tree-like' regular graphs~\cite{maj}. Some very recent results on bootstrap percolation in two dimensions can be found in~\cite{GHM} and~\cite{DCH}. For more on the links between bootstrap percolation and statistical physics, see~\cite{FSS,NNS,NS}, or the survey~\cite{braz}.

The main step in the proof of Theorems~\ref{hypercube} and~\ref{n^d} is Theorem~\ref{2^ell}, below, which gives very tight bounds on the probability that a small hypercube (a `critical droplet') is internally spanned (see below). In order to state this result, however, we first need the following definitions.

Let $d$ and $n = n(d)$ be fixed, and consider two-neighbour bootstrap percolation on $[n]^d$; that is, from now on we assume that $G = [n]^d$ and $r = 2$. We shall call a set which is isomorphic to $[a_1] \times \dots \times [a_d]$ for some $a_1,\ldots,a_d \in \N$ a \emph{cube}, and a set which is isomorphic to $[2]^\ell$ for some $\ell \in \N_0$ a \emph{hypercube}. Note that given any set $A \subset [n]^d$, the set $[A]$ is a disjoint union of cubes at distance at least 3 from each other.

Given a cube $Q \cong [a_1] \times \ldots \times [a_d] \subset [n]^d$, define the \emph{dimension} of $Q$ to be
$$\dim(Q) \; := \; \ds\sum_{i=1}^d \big( a_i - 1 \big),$$
and note that a single vertex is a cube of dimension zero. We write $d(S,T)$ for the distance (in the graph $[n]^d$) between two sets $S$ and $T$.

We say that the cube $Q$ is \emph{internally spanned} by a set $A \subset [n]^d$ if $[A \cap Q] = Q$. For hypercubes, the following more restrictive definition will play a crucial role in the proof.

\begin{defn}
Let $\ell \in \N_0$. A cube $Q \cong [2]^{2\ell}$ is said to be \emph{sequentially spanned} by a set $A \subset Q$ if $|A| = \ell + 1$, and there exists an order $(a_0,\ldots,a_\ell)$ of the elements of $A$ such that
$$d(a_{j+1},[\{a_0,\ldots,a_j\}]) = 2$$
for each $0 \le j \le \ell - 1$. In particular, note that $[\{a_0,\ldots,a_j\}]$ is a subcube of dimension $2j$ for each $j \in [\ell]$. We call such an ordering $(a_0,\ldots,a_\ell)$ a \emph{spanning sequence} for $Q$.

We say that $Q \cong [2]^{2\ell}$ is \emph{sequentially internally spanned} by a set $A \subset [n]^d$ if it is sequentially spanned by $A \cap Q$.
\end{defn}

Given a set $S$, we write $A \sim \Bin(S,p)$ to mean that the elements of $A \subset S$ are chosen independently at random with probability $p$. Recall that $\Pr_p$ denotes the distribution $A \sim \Bin([n]^d,p)$. Now, for each $\ell \in \N_0$ define
$$P(\ell,p) \; := \; \Pr_p\big( A \textup{ internally spans } [2]^\ell \big),$$ and
$$Q(2\ell,p) \; := \; \Pr_p\big( A \textup{ sequentially internally spans } [2]^{2\ell} \big).$$
Balogh and Bollob\'as~\cite{BB} proved that, if $2^\ell p \le 1$, then
$$\ell^\ell 2^{\ell^2/4 - 2\ell} p^{(\ell + 3)/2} \; \le \; P(\ell,p) \; \le \; \ell^\ell 2^{\ell^2/4} p^{(\ell+2)/2},$$
and used this result to determine $p_c([2]^d,2)$ up to a constant factor. An error of order $2^{\Theta(\ell)}$ in the approximation of $P(\ell,p)$ corresponds to a multiplicative constant error in the calculation of $p_c$. Thus for Theorem~\ref{n^d} we need to determine $P(\ell,p)$ up to a factor of $2^{o(\ell)}$, and for Theorem~\ref{hypercube} we are allowed only a factor of $\ell^{O(\log \ell)}$. In fact we shall do somewhat better even than this. The following theorem determines $P(\ell,p)$ up to a \emph{constant} multiplicative factor; for even $\ell$ this constant is $5/2$.

\begin{thm}\label{2^ell}
There exists a constant $\delta > 0$ such that the following holds. Let $\ell \in \N$ and $p > 0$, with $\ell^2 2^{2\ell}p \le \delta$, and let $\lambda \approx 1.166$ be as defined above. Then
\begin{enumerate}
\item[$(a)$] $\ds\frac{2}{5} \, (2\ell)! \, \lambda^{-\ell} \, 2^{\ell^2} p^{\ell + 1} \; \le \; Q(2\ell,p) \; \le \; P(2\ell,p) \; \le \; (2\ell)! \, \lambda^{-\ell} \, 2^{\ell^2} p^{\ell + 1}$.\\
\item[$(b)$] $\ds\frac{1}{100} \, (2\ell + 1)! \, \lambda^{-\ell} \, 2^{(\ell + 1)^2} p^{\ell + 2} \; \le \; P(2\ell + 1,p) \; \le \; 5 (2\ell + 1)! \, \lambda^{-\ell} \, 2^{(\ell + 1)^2} p^{\ell + 2}$.
\end{enumerate}
\end{thm}

It is interesting to note that, since our proof is by induction, and is extremely delicate in places, we \emph{could not} have proved a much weaker result than Theorem~\ref{2^ell}. In particular, our proof does not work if the induction hypothesis has a large multiplicative constant error term, except in the lower bound of part $(b)$.

The rest of the paper is structured as follows. First, in Section~\ref{toolsec} we state a technical lemma, and make some definitions. Section~\ref{2^dsec} contains the proof of Theorem~\ref{2^ell} and is the main part of the paper. In Sections~\ref{uppersec} and~\ref{lowersec} we shall deduce Theorems~\ref{hypercube} and~\ref{n^d} from Theorem~\ref{2^ell}, and in Section~\ref{probsec} we state some open problems. Finally, in Appendix A we prove the technical lemma.

\section{Tools and Notation}\label{toolsec}

In this section we shall state our main technical lemma, and recall some basic facts about bootstrap percolation. The proof of Lemma~\ref{tech} is somewhat lengthy, and is given in Appendix A.

For each $m \in \N$, let
\begin{equation}\label{adef}
a_m = \ds\frac{(2\lambda)^m}{2^{m^2}m!},
\end{equation}
and note that $\ds\sum_{m=1}^\infty (-1)^{m+1} a_m = 1$, by the definition of $\lambda$ (see~\eqref{xdef}, above).

\begin{lemma}\label{tech}
Let $\ell \in \N$, and let $g,h: \N_0 \to \RR^+$ be non-negative functions satisfying $g(0) = 0$, and $1 \le h(m) \le 1 + g(m)$ for every $m \in \N_0$, and let $\lambda \approx 1.166$ and $(a_m)$ be as defined above. Suppose that $f(0) = 1$, $f(1) = \lambda/2$, and
$$f(t) = \ds\sum_{m=1}^{t} (-1)^{m+1} a_m h(t-m) f(t-m)$$ for every $2 \le t \le \ell$. Then
$$1 \, - \, \frac{\lambda}{2} \; \le \; f(\ell) \; \le \; \frac{\lambda}{2} \, \exp\left( \frac{1}{2- \lambda} \ds\sum_{m=1}^{\ell-1} g(m) \right).$$
\end{lemma}

\begin{rmk}
The bounds in Lemma~\ref{tech} are close to best possible. To see this, suppose (for simplicity) that we replace the sequence $(a_1,a_2,\ldots)$ by the (ideal) sequence $(1,0,0,\ldots)$. Then $f(\ell) = \lambda/2$ if $h(m) = 1$ for every $m \in \N$, and
$$f(\ell) \; = \; \frac{\lambda}{2} \prod_{m=1}^{\ell-1} \Big( 1 + g(m) \Big) \; \le \; \frac{\lambda}{2} \exp\left( \sum_{m=1}^{\ell-1} g(m) \right)$$
if $h(m) = 1 + g(m)$ for every $m \in \N$.
\end{rmk}

We next recall the concept of \emph{disjoint occurrence} of events, and the van den Berg-Kesten Lemma~\cite{vBK}, which utilizes it. In the setting of bootstrap percolation on a graph $G$, two increasing events $E$ and $F$ occur disjointly if there exist disjoint sets $S,T \subset V(G)$ such that the infected sites in $S$ imply that $E$ occurs, and the infected sites in $T$ imply that $F$ occurs. (We call $S$ and $T$ \emph{witness sets} for $E$ and $F$.) We write $E \circ F$ for the event that $E$ and $F$ occur disjointly.

\begin{vBKlemma}
Let $E$ and $F$ be any two increasing events defined in terms of the infected sites $A \subset V(G)$, and let $p \in (0,1)$. Then
$$\Pr_p(E \circ F) \; \le \; \Pr_p(E)\,\Pr_p(F).$$
\end{vBKlemma}

We next state a fundamental lemma from~\cite{BB}, which allows us to apply the van den Berg-Kesten Lemma. An almost identical lemma was proved independently by Holroyd~\cite{Hol} for the two-dimensional lattice.

\begin{lemma}\label{ST}
Let $Q \subset [n]^d$ be a cube, and suppose $Q$ is internally spanned by $A \subset [n]^d$. Then there exist proper subcubes $S, T \subsetneq Q$ such that
\begin{itemize}
\item $[S \cup T] = Q$, and
\item $S$ and $T$ are disjointly internally spanned by $A$.
\end{itemize}
In particular, we may take $S$ to be the largest internally spanned proper subcube of $Q$.
\end{lemma}

We give a sketch of the proof to aid the reader who is unfamiliar with these concepts; for the details, see the proof of Lemma~\ref{crossk}, which is a slight generalization.

\begin{proof}[Sketch of proof]
Given two subcubes, $S$ and $T$, the span of their union, $[S \cup T]$ is either just their union, $S \cup T$, or a new (larger) subcube. Thus, if $A$ spans $Q$, we may consider the process as a series of moves of the form $(S,T) \mapsto [S \cup T]$, where $[S \cup T]$ is a cube, starting with all single sites (dimension zero subcubes), and ending up with $Q$. At each stage, all of the subcubes are disjointly spanned. The subcubes $S$ and $T$ given by the lemma are those found at the penultimate step of this process.
\end{proof}

The next lemma is also from~\cite{BB}, and gives the minimal number of sites in a spanning configuration. Given a cube $Q = [a_1] \times \ldots \times [a_d] \subset [n]^d$, recall that the \emph{dimension} of $Q$ is
$$\dim(Q) \; = \; \ds\sum_{i=1}^d \big( a_i - 1 \big).$$
Note that if $S$, $T$ and $Q$ are cubes, and $[S \cup T] = Q$, then $\dim(S) + \dim(T) \ge \dim(Q) - 2$. This follows because the distance between $S$ and $T$ must be at most two (since their closure is a cube), and so they must overlap in all but at most two directions.

\begin{lemma}\label{minl}
Let $Q$ be a subcube of $[n]^d$. If $A$ spans $Q$, then
$$|A| \; \ge \; \ds\frac{\dim(Q)}{2} + 1.$$
In particular, if $A$ spans $[2]^{2\ell}$ then $|A| \ge \ell + 1$, and if $A$ spans $[2]^{2\ell + 1}$ then $|A| \ge \ell + 2$.
\end{lemma}

\begin{proof}
The proof is by induction on the dimension of $Q$, using Lemma~\ref{ST}. Indeed, let $S$ and $T$ be disjointly internally spanned subcubes with $[S \cup T] = Q$, and let $A_1,A_2 \subset A$ with $[A_1] = S$, $[A_2] = T$ and $A_1 \cap A_2 = \emptyset$. Since $[S \cup T] = Q$, we have $\dim(S) + \dim(T) \ge \dim(Q) - 2$. Hence
$$|A \cap Q| \; \ge \; |A_1| + |A_2| \; \ge \; \ds\frac{\dim(S) + \dim(T)}{2} + 2 \; \ge \; \ds\frac{\dim(Q)}{2} + 1,$$
as required.
\end{proof}

We finish this section by describing some of the notation we shall use throughout the paper, and especially in Section~\ref{2^dsec}. Given a subcube $Q \subset [2]^d$, we shall denote $Q$ as a member of $\{0,1,*\}^d$ in the obvious way, i.e., if $Q = (z_1,\ldots,z_d) \in \{0,1,*\}^d$, then
$$(y_1,\ldots,y_d) \in Q \; \Leftrightarrow \; \textup{$y_j = z_j$ for every $j$ with } z_j \in \{0,1\}.$$
We shall denote by $Q\langle j_1,\ldots,j_k \rangle$ the collection of maximal subcubes of $Q$ which are constant on $\{j_1,\ldots,j_k\}$, and by $Q[j_1,\ldots,j_k]$ the collection of maximal subcubes of $Q$ which are constant on $\{j_1,\ldots,j_k\}^c = [d] \setminus \{j_1,\ldots,j_k\}$.

Given subcubes $B,C \subset [2]^d$, we write $\Delta(B,C)$ for the set of directions in which $B$ and $C$ are both constant and differ, and given vertices $b,c \in [2]^d$, define $\Delta(b,c)$ and $\Delta(b,C)$ similarly, noting that a vertex is a cube of dimension zero. Observe that $d(B,C) = |\Delta(B,C)|$. Finally, note that $\N = \{1,2,\ldots\}$ and $\N_0 = \{0,1,2,\ldots\}$.

\section{Percolation in $[2]^d$}\label{2^dsec}

In this section we shall prove Theorem~\ref{2^ell}. The proof comes in two parts: first we bound $Q(2\ell,p)$ from both sides, and then we bound $P(\ell,p)$ from above. The methods in each case will be similar, but for $P(\ell,p)$ we shall have many more complications to overcome.

\subsection{Bounding $Q(2\ell,p)$: a lower bound for $P(2\ell,p)$}\label{Qsec}

We begin with the simpler case; not only will it serve as a helpful warm-up, but also this result will be a crucial tool in the proof of the harder case, and in Section~\ref{uppersec} (see in particular Lemmas~\ref{Yeven},~\ref{Yodd} and~\ref{varX}, and Proposition~\ref{oddlower}).

Let $\ell \in \N_0$, and suppose that the cube $C = [2]^{2\ell}$ initially contains exactly $\ell + 1$ active sites. Let $\A(\ell)$ denote the collection of all such $(\ell + 1)$-sets, and let $\S(\ell) \subset \A(\ell)$ denote the collection of those $(\ell + 1)$-sets which sequentially span $C$. We shall prove the following theorem.

\begin{thm}\label{Q}
Let $\ell \in \N$, and let $\lambda \approx 1.166$ be as defined in~\eqref{xdef}. Then,
$$\left(1 \, - \, \frac{\lambda}{2} \right) (2\ell)! \, \lambda^{-\ell} \, 2^{\ell^2} \; \le \; |\S(\ell)| \; \le \; \frac{\lambda}{2} \, (2\ell)! \, \lambda^{-\ell} \, 2^{\ell^2}.$$
Hence, if $2^{2\ell}p$ is sufficiently small, then
$$\frac{2}{5} (2\ell)! \, \lambda^{-\ell} \, 2^{\ell^2} p^{\ell + 1}\; \le \; Q(2\ell,p) \; \le \; \frac{3}{5} (2\ell)! \, \lambda^{-\ell} \, 2^{\ell^2} p^{\ell + 1}.$$
\end{thm}

We start by covering $\S(\ell)$ with sets $N(j,k)$ as follows. For each $A \in \S(\ell)$ and each $1 \le j < k \le 2\ell$, say that the pair $\{j,k\}$ is \emph{an ending for} $A$ if some $A' \subset A$ sequentially spans one of the subcubes $C\<j,k\>$. Note that if $\{j,k\}$ is an ending for $A$, then $j$ and $k$ are the last dimensions to be infected for some spanning sequence for $C$, and that $|A \setminus A'| = 1$.

Now, for each $1 \le j < k \le 2\ell$, define
$$N(j,k) \; = \; N_{\ell}(j,k) \; = \; \big\{A \in \S(\ell) : \{j,k\} \textup{ is an ending for } A \big\},$$
and observe that $\S(\ell) = \ds\bigcup_{j < k} N(j,k)$. We shall count $\S(\ell)$ by inclusion-exclusion, using the sets $N(j,k)$. The following lemma is the key step.

\begin{lemma}\label{Njk}
Let $\ell, m \in \N$ with $\ell \ge 2$, and let $\{j_1,k_1\},\ldots,\{j_m,k_m\} \subset [2\ell]$ be distinct pairs. If the elements $j_q$ and $k_q$ are all distinct, then
$$\left| \bigcap_{q=1}^{m} N_{\ell}(j_q,k_q) \right| \; = \; 2^{2m} \left( 2^{2\ell - 2m} \right)^m |\S(\ell - m)|,$$
and otherwise $\bigcap_{q=1}^{m} N_{\ell}(j_q,k_q) = \emptyset$.
\end{lemma}

We shall use the following two simple lemmas in the proof of Lemma~\ref{Njk}. We shall prove a slightly more general version of each that we need here, since we shall use them again in the proof of Lemma~\ref{Mjk} in Section~\ref{RRsec}. We need the following definition.

\begin{defn}
Given a cube $Q \subset [2]^d$ and a set $A \subset Q$, we say that $\{j,k\}$ is a \emph{final pair} for $A$ in $Q$ if $A$ internally spans some subcube $C \in Q\<j,k\>$, and internally spans $Q$, and $|A \setminus C| = 1$. We say that $\{i\}$ is a \emph{final element} for $A$ in $Q$ if $A$ internally spans one of the subcubes $C \in Q\<i\>$, and internally spans $Q$, and $|A \setminus C| = 1$.
\end{defn}

Note that if $\{j,k\}$ is an ending for $A$, then it is also a final pair for $A$ in $Q$ (but not vice-versa, since the former has the extra condition that $C$ is sequentially spanned). The following lemma says that final pairs and final elements are all disjoint.

\begin{lemma}\label{finalsdj}
Let $\{j,k\}$ be a final pair for $A$ in $Q$. Then $j$ and $k$ are not final elements for $A$ in $Q$. Moreover, if $\{j',k'\}$ is a different final pair, then $j$, $k$, $j'$ and $k'$ are all distinct.
\end{lemma}

\begin{proof}
First suppose that both $\{j,k\}$ and $\{j,k'\}$ are final pairs for $A$ in $Q$, and for simplicity assume that $j = 1$, $k = 2$ and $k' = 3$. Then $A$ internally spans a subcube in $Q\<1,2\>$, without loss of generality suppose it is $C = (0,0,*,\ldots,*)$, and $A \setminus C$ consists of a single element, which lies in the cube $(1,1,*,\ldots,*)$. In particular, note that the subcubes $(0,1,*,\ldots,*)$ and $(1,0,*,\ldots,*)$ contain no elements of $A$.

Now, $A$ also internally spans some subcube $D$ in $Q\<1,3\>$. But one of the sets $D \cap (0,1,*,\ldots,*)$ and $D \cap (1,0,*,\ldots,*)$ is a subcube of $D$, of dimension $\dim(D) - 1$, which contains no element of $A$ (the other is the empty set). Thus $A$ cannot span $D$, and so we have a contradiction. It is similarly easy to see that if $\{j,k\}$ is a final pair, then $\{j\}$ is not a final element for $A$ in $Q$.
\end{proof}

The next lemma will allow us to recursively obtain a sequence of sequentially internally spanned subcubes of a sequentially spanned hypercube $Q$.

\begin{lemma}\label{addlast}
Let $\ell \in \N$, let $A \subset Q = [2]^\ell$, and suppose that $A$ spans $Q$. Let $S$ be an $(\ell-2)$-dimensional subcube of $Q$, with $|A \setminus S| = 1$. Then $S$ is internally spanned by $A$. Moreover, if $Q$ is sequentially spanned by $A$, then $S$ is sequentially spanned by $A \cap S$.
\end{lemma}

\begin{proof}
Let $S \subset Q$ be a cube with $\dim(S) = \ell - 2$, and let $v$ be the unique element of $A \setminus S$. We may assume, without loss of generality, that $S = (0,0,*,\dots,*)$, and that $v = (1,1,0,\dots,0)$. The basic fact we shall use is that, for any subcube $T \subset S$,
$$[T \cup \{v\}] \cap S \; = \; T,$$
which follows because if $d(v,T) \le 2$ then we must have $(0,\ldots,0) \in T$.

Now suppose $A$ spans $Q$, and let $[A \cap S] = T_1 \cup \ldots \cup T_t$, where the $T_j$ are all subcubes of $S$, and $d(T_i,T_j) \ge 3$ for every $i \neq j$. If $A \cap S$ does not span $S$ then $t \ge 2$. Note that $d(T_j,v) \le 2$ for at most one $j \in [t]$, and that $[T_j \cup \{v\}] \cap S = T_j$. Thus $d([T_i],[T_j \cup \{v\}]) \ge 3$ for every $i \neq j$. Therefore $A$ does not percolate in $Q$, which is a contradiction.

Next suppose that $A$ sequentially spans $Q$, let $(a_0,\ldots,a_\ell)$ be a spanning sequence for $Q$, so
$$d\big( a_{j+1},\big[ \{ a_0,\ldots,a_j \} \big] \big) = 2$$
for each $0 \le j \le \ell - 1$, and suppose that $v = a_q$. We claim that
$$(a_0,\ldots,a_{q-1},a_{q+1},\ldots,a_\ell)$$
is a spanning sequence for $S$. Indeed, let $U_j = [\{a_0,\ldots,a_j\} \setminus \{a_q\}]$ for each $j \in [\ell]$, and note that $U_j \subset S$ and $[U_j \cup \{a_q\}] \cap S = U_j$. Also $a_{j+1} \in S$ for every $j + 1 \neq q$, so $d(a_{j+1},U_j) = d(a_{j+1},[U_j \cup \{a_q\}]) = 2$. Therefore $A \setminus \{a_q\}$ sequentially spans $S$, as claimed.
\end{proof}

The following structural lemma follows from Lemma~\ref{addlast}. We shall use this result in the proof of Lemma~\ref{Njk}, below, and also in Section~\ref{RRsec}.

\begin{lemma}\label{structure}
Let $\ell,m \in \N$, with $\ell \ge 3$ and $\ell \ge 2m$, and let $A \subset Q = [2]^\ell$. Suppose that $A$ spans $Q$, and that $\{2q - 1,2q\}$ is a final pair for $A$ in $[2]^\ell$ for each $q \in [m]$.

Then there exists a subcube $C' \in C\<1,\ldots,2m\>$, and a set $A' = A \setminus \{a_1,\ldots,a_m\}$, such that
\begin{itemize}
\item[$(a)$] $A'$ spans $C'$, \smallskip
\item[$(b)$] $\Delta(a_q,C') = \{2q-1,2q\}$ for each $q \in [m]$.
\end{itemize}
 \smallskip
Moreover, if $\{2q-1,2q\}$ is an ending for each $q \in [m]$, then $A'$ sequentially spans $C'$.
\end{lemma}

\begin{proof}
For any final pair $\{j,k\}$ for $A$, there is a vertex $a \in A$ such that $j,k \in \Delta(a,c)$ for every $a \neq c \in A$. Since $|A| \ge \lceil \ell/2 \rceil + 1 \ge 3$, this element is unique. Let $a_q$ be this vertex for the pair $\{2q-1,2q\}$ for each $q \in [m]$.

First note that the $a_q$ are all distinct. Indeed, if $a_1 = a_2$ say, then the other elements of $A$ must all lie in a $(2\ell-4)$-dimensional subcube of $C$ at distance 4 from $a_1$, and so $A$ would not span $C$.

Thus we may apply Lemma~\ref{addlast} for each vertex $a_q$ in turn, obtaining a sequence of subcubes $C = C_0 \supset C_1 \supset \ldots \supset C_m$, such that $\dim(C_q) = 2\ell - 2q$ for $q \in [m]$, and $A \cap C_q$ spans $C_q$. Let $C' = C_m$, and $A' = A \cap C'$.

Finally, observe that $\Delta(a_q,C') = \{2q-1,2q\}$ for each $q \in [m]$. This follows because $A'$ is non-empty, so for any $q \neq r \in [m]$ we have $2r-1,2r \in \Delta(a_r,a') \cap \Delta(a_r,a_q)$ for any $a' \in A'$, which means that $2r-1,2r \notin \Delta(a_q,a')$, and so $2r-1,2r \notin \Delta(a_q,C')$.

The proof that $A'$ sequentially spans $C'$, if each $\{2q-1,2q\}$ is an ending, is exactly the same, using the corresponding part of Lemma~\ref{addlast}.
\end{proof}

We can now prove Lemma~\ref{Njk}.

\begin{proof}[Proof of Lemma~\ref{Njk}]
First note that if $A \in N(j,k)$ then $\{j,k\}$ is a final pair for $A$ in $C = [2]^{2\ell}$, so it follows that $N(j,k) \cap N(j,k') = \emptyset$ for $k \neq k'$ by Lemma~\ref{finalsdj}. So let $j_1,\ldots,j_m, k_1,\ldots,k_m$ be distinct elements of $[2\ell]$, and assume for simplicity that in fact $j_q = 2q - 1$ and $k_q = 2q$ for each $q \in [m]$.

Suppose that $A \in \bigcap_{q=1}^{m} N(j_q,k_q)$, so $\{2q-1,2q\}$ is an ending for $A$ for each $q \in [m]$. Then, by Lemma~\ref{structure}, there exists a subcube $C' \in C\<1,\ldots,2m\>$, and a subset $A' \subset A$, such that $A'$ sequentially spans $C'$, and $A \setminus A' = \{a_1, \ldots, a_m\}$, for some $a_q$ with $\Delta(a_q,C') = \{2q-1,2q\}$ for each $q \in [m]$.

Now we only have to count. We have $2^{2m}$ choices for $C'$ and, given $C'$, $|\S(\ell - m)|$ choices for $A'$ and $2^{2\ell - 2m}$ choices for each $a_q$. Finally, we claim that each of these choices gives a different set $A$. Indeed, if $\ell > m$ then $\dim(C') \ge 2$, and so $C'$ is the only member of $C\<1,\ldots,2m\>$ with more than one element (the others are either empty, or contain exactly one $a_q$). Thus we can reconstruct $C'$, $A'$ and $a_1,\ldots,a_m$ from $A$. If $\ell = m$ then $C'$ is a single vertex, but is the only element of $A$ at distance two from all of the others, since $d(a_q,a_r) \ge 4$ for each $q,r \in [m]$. Note that here we needed $m = \ell \ge 2$.
\end{proof}

\begin{obs}\label{noofpairs}
Let $\ell,m \in \N$ with $\ell \ge m$. There are
$$\ds\frac{1}{m!} \, \ds\prod_{q=0}^{m-1} {{2\ell - 2q} \choose 2} \; = \; \frac{(2\ell)!}{2^m m!(2\ell-2m)!}$$
different ways of choosing $m$ disjoint pairs from $[2\ell]$.
\end{obs}

Now, using Lemmas~\ref{tech} and \ref{Njk}, we can prove Theorem~\ref{Q}.

\begin{proof}[Proof of Theorem~\ref{Q}]
The proof is by the technical lemma, Lemma~\ref{tech}. Define a function $g: \N_0 \to \RR^+$ by
$$|\S(\ell)| \; = \; g(\ell) (2\ell)! \, \lambda^{-\ell} \, 2^{\ell^2},$$
for each $\ell \in \N_0$, and observe that therefore
$$\left( \prod_{q=0}^{m-1} {{2\ell - 2q} \choose 2} \right) 2^{2\ell m - m^2 + m} |\S(\ell - m)| \; = \; g(\ell - m) (2\ell)! \, \lambda^{- \ell + m} \, 2^{\ell^2}.$$
Recall that $\S(\ell) = \ds\bigcup_{j < k} N(j,k)$. Therefore, by inclusion-exclusion, Lemma~\ref{Njk} and Observation~\ref{noofpairs}, we have
\begin{eqnarray*}
|\S(\ell)| & = & \sum_{m=1}^\ell (-1)^{m+1} \sum_{\{j_1,k_1\},\ldots,\{j_m,k_m\}} \left| \bigcap_{q=1}^{m} N_{\ell}(j_q,k_q) \right|\\
& = & \sum_{m=1}^\ell (-1)^{m+1} \frac{1}{m!} \left( \prod_{q=0}^{m-1} {{2\ell - 2q} \choose 2} \right) 2^{2\ell m - 2m^2 + 2m} |\S(\ell - m)|,\\
& = &  \sum_{m=1}^\ell (-1)^{m+1} \frac{2^{-m^2 + m}}{m!}  \left( g(\ell - m) (2\ell)! \, \lambda^{- \ell + m} \, 2^{\ell^2} \right)\\
& = & (2\ell)! \, \lambda^{-\ell}\, 2^{\ell^2} \sum_{m=1}^\ell (-1)^{m+1} \left( \frac{(2\lambda)^{m}}{2^{m^2} m!} \right) g(\ell - m).
\end{eqnarray*}
So, recalling that $a_m = \ds\frac{(2\lambda)^{m}}{2^{m^2} m!}$, it follows that
\begin{equation} g(\ell) = \sum_{m=1}^\ell (-1)^{m+1} a_m g(\ell - m) \label{Qrec} \end{equation}
for every $2 \le \ell \in \N$. Moreover, note that $|\S(0)| = 1$ and $|\S(1)| = 2$, so it is easy to check that $g(0) = 1$ and $g(1) = \lambda/2$. Thus, by Lemma~\ref{tech}, applied with $h(m) = 1$, we obtain
$$1 \,- \,\frac{\lambda}{2} \; \le \; g(\ell) \; \le \; \frac{\lambda}{2},$$ as required.

Finally note that, for each $A \in \S(\ell)$, the probability that the set of initially active sites in $[2]^{2\ell}$ is exactly $A$ is $p^{\ell+1}(1-p)^{2^{2\ell} - \ell - 1}$. Hence
$$1 \, - \, \frac{\lambda}{2} \; \le \; \frac{Q(2\ell,p)}{(2\ell)! \, \lambda^{-\ell} \, 2^{\ell^2} p^{\ell + 1} (1 - p)^{2^{2\ell} - \ell - 1}} \; \le \; \frac{\lambda}{2}.$$
The bounds for $Q(2\ell,p)$ now follow if $2^{2\ell}p$ is sufficiently small, since $\lambda < 6/5$.
\end{proof}

\begin{rmk}
Theorem~\ref{Q} determines $|\S(\ell)|$ up to a factor of about $3/2$, but in fact the proof allows us to calculate $|\S(\ell)|$ exactly for any fixed $\ell$, via the recursion (\ref{Qrec}). Computer simulation indicates that the correct constant $g(\ell)$ is approximately $0.4976$ for large $\ell$.
\end{rmk}

\subsection{Bounding $P(\ell,p)$: motivation and definitions}\label{Pdefsec}

We now move on to the harder part of the proof: bounding $P(\ell,p)$ from above. There are several ways in which $[2]^\ell$ may be spanned, and in order to deal with them (by induction on $\ell$) we shall need very sharp bounds on $P(\ell,p)$. We overcome these problems by defining a new function, $R(\ell,p)$, for which we obtain a precise recurrence relation (see Lemma~\ref{R2ell}). We then prove that $R(\ell,p) \approx P(\ell,p)$ (see Lemma~\ref{PandR}) using the induction hypothesis, which allows us to apply the technical lemma. We deduce that sequential spanning is (roughly) the dominant kind of internal spanning when $\ell$ is even and $p$ is not too large.

Let us begin by defining a collection of events $M(j,k)$, which will serve the same purpose here as the sets $N(j,k)$ did above. Recall that, given a cube $Q$ and a set $A \subset Q$, we say that $\{j,k\}$ is a \emph{final pair} for $A$ in $Q$ if $A$ internally spans $Q$, $A$ internally spans some subcube $C \in Q\<j,k\>$, and $|A \setminus C| = 1$, and that $\{i\}$ is a \emph{final element} for $A$ in $Q$ if $A$ internally spans $Q$, $A$ internally spans one of the subcubes $C \in Q\<i\>$, and $|A \setminus C| = 1$.

In all that follows, we shall assume that the elements of $A$ are chosen independently at random with probability $p$. Now given $\ell \in \N$, and $i,j,k \in [\ell]$ with $j \neq k$, let
$$L(i) \; = \; L_\ell(i) \; := \; \textup{the event that }\{i\}\textup{ is a final element for $A$ in $[2]^\ell$},$$
and
$$M(j,k) \; = \; M_\ell(j,k) \; := \; \textup{the event that }\{j,k\}\textup{ is a final pair for $A$ in $[2]^\ell$},$$
and note that $P(\ell,p) \ge \Pr_p\left( \ds\bigcup_{j<k} M(j,k) \cup \bigcup_i L(i) \right)$.

Now, for each $\ell \in \N_0$ and $p \in (0,1)$, let
$$R(2\ell,p) \; := \; \Pr_p\left( \bigcup_{j<k} M_{2\ell}(j,k) \right)$$
and let
$$R(2\ell+1,p) \; := \; \Pr_p\left( \bigcup_{j<k} M_{2\ell+1}(j,k) \cup \bigcup_i L_{2\ell+1}(i) \right)$$
and note that $Q(2\ell,p) \le R(2\ell,p)$, and that $R(\ell,p) \le P(\ell,p)$ for each $\ell \in \N_0$.

We now state the theorem we shall in fact prove by induction on $\ell$.

\begin{thm}\label{induction}
Let $\delta > 0$ be sufficiently small. Then there exists a non-negative function $g : \N_0 \to \RR^+$, with $\ds\sum_{m=0}^\infty g(m) \le \frac{1}{3}$, such that
$$R(2\ell,p) \; \le \; \frac{\lambda}{2} \exp\left( \frac{1}{2- \lambda} \sum_{m=0}^{\ell-1} g(m) \right) (2\ell)! \, \lambda^{-\ell} \, 2^{\ell^2}  p^{\ell+1} (1-p)^{2^{2\ell} - \ell - 1}$$
for every $\ell \in \N$ and $p > 0$ with $\ell^2 2^{2\ell} p \le \delta$.
Also, under the same conditions,
$$R(2\ell+1,p) \; \le \; 4 (2\ell + 1)! \, \lambda^{-\ell} \, 2^{(\ell+1)^2} p^{\ell+2}.$$
Moreover,
$$P(2\ell,p) \; \le \; \big( 1 + g(\ell) \big) R(2\ell,p)$$
and
$$P(2\ell+1,p) \; \le \; 5 (2\ell + 1)! \, \lambda^{-\ell} \, 2^{(\ell+1)^2} p^{\ell+2}$$
for every $\ell \in \N$.
\end{thm}

The upper bounds in Theorem~\ref{2^ell} follow immediately from Theorem~\ref{induction}. Indeed, first note that
$$1 + g(\ell) \; \le \; e^{g(\ell)} \; \le \; \exp\left( \frac{g(\ell)}{2 - \lambda} \right),$$
since $\lambda > 1$ and $g(\ell) \ge 0$. Hence
\begin{eqnarray*}
P(2\ell,p) \; \le \; \frac{\lambda}{2} \exp\left( \frac{1}{2- \lambda} \sum_{m=0}^\ell g(m) \right) (2\ell)! \, \lambda^{-\ell} \, 2^{\ell^2} \, p^{\ell+1} \, (1-p)^{2^{2\ell} - \ell - 1},
\end{eqnarray*}
and
$$\frac{\lambda}{2} \exp\left( \frac{1}{2- \lambda} \sum_{m=0}^\ell g(m) \right) \; \le \; \frac{\lambda}{2} \exp\left( \frac{1}{3(2- \lambda)} \right) \; < \; 1,$$
since $\lambda < 6/5$. We remark that with a little more care, we could have improved this upper bound by a factor of about $4/5$.

\bigskip
We shall need the induction hypothesis in order to prove several of our preliminary lemmas. So, for ease of reference, we shall refer to the following conditions as the \emph{properties $(*)$ for $L$.}
\begin{enumerate}
\item[$(a)$] If $2\ell \le L$ then
$$P(2\ell,p) \; \le \; (2\ell)! \, \lambda^{-\ell} \, 2^{\ell^2} p^{\ell+1} \, (1-p)^{2^{2\ell} - \ell - 1}.$$
\item[$(b)$] If $2\ell + 1 \le L$ then
$$P(2\ell+1,p) \; \le \; 5 (2\ell + 1)! \, \lambda^{-\ell} \, 2^{(\ell+1)^2} p^{\ell+2}.$$
\end{enumerate}
Observe that the properties $(*)$ for $\ell$ imply that
\begin{equation}\label{genupper}
P(\ell,p) \; \le \; \ell! \, \lambda^{-\ell/2} \, 2^{\ell^2/4} p^{\ell/2+1}
\end{equation}
if $2^\ell p \le 1/100$, say.

\subsection{Recurrence relations}\label{RRsec}

We shall bound $R(\ell,p)$ using inclusion-exclusion, induction, and the technical lemma. We shall have to treat the odd and even cases separately. The following lemma gives us the recursions we need; combining them with Lemma~\ref{PandR} (below) will give the desired result.

For every $\ell \in \N_0$, let
\begin{equation}\label{zdef0}
z(2\ell) \; = \; (2\ell)! \,\lambda^{-\ell} \,2^{\ell^2}.
\end{equation}
Also, here and throughout, let
\begin{equation}\label{Edef}
\E_{\ell,m,p} \; := \; (1 - p)^{2^\ell - 2^{\ell-2m} - m}.
\end{equation}

\begin{lemma}\label{R2ell}
Let $2 \le \ell \in \N$ and $p > 0$. Then
$$R(2\ell,p) \; = \; \sum_{m \ge 1} (-1)^{m+1} a_m \,p^m \,\E_{2\ell,m,p} \frac{z(2\ell)}{z(2\ell-2m)} P(2\ell - 2m,p),$$
and
\begin{eqnarray*}
R(2\ell+1,p) & \le & (2\ell+1)\,2^{2\ell+1} \,p \,P(2\ell,p) \,+\, {{2\ell+1} \choose 2}  2^{2\ell+1} \, p \, P(2\ell -1, p).
\end{eqnarray*}
\end{lemma}

Lemma~\ref{R2ell} will follow fairly easily from the following lemma, which plays the same role in the current proof as did Lemma~\ref{Njk} in the previous subsection. We remark that we could obtain a similar recurrence for $\bigcap_q M(j_q,k_q) \cap L(i)$, but we shall not need this. However, the situation for $L(i) \cap L(j)$ is more complicated (see Proposition~\ref{oddlower}).

\begin{lemma}\label{Mjk}
Let $\ell \ge 3$ and $m \ge 1$ be integers, let $\{j_1,k_1\},\ldots,\{j_m,k_m\} \subset [\ell]$ be distinct pairs, and let $p > 0$. If the elements $j_q$ and $k_q$ are all distinct, then
$$\Pr_p\left(\bigcap_{q=1}^m M(j_q,k_q) \right) \; = \; 2^{2m} \left( 2^{\ell - 2m} \right)^m p^m \,\E_{\ell,m,p} \,P(\ell - 2m, p),$$
and otherwise $\Pr_p\left( \bigcap_{q=1}^m M_\ell(j_q,k_q) \right) = 0$.
\end{lemma}

\begin{proof}
The proof is almost the same as that of Lemma~\ref{Njk}. First note that we have
$$\Pr_p\big( M(j,k) \cap M(j,k') \big) \; = \; 0$$
if $k \neq k'$ by Lemma~\ref{finalsdj}.

So let $C = [2]^\ell$, let $\{j_1,k_1,\ldots,j_m,k_m\}$ be a set of $2m$ distinct elements of $[\ell]$, and assume for simplicity that in fact $j_q = 2q - 1$ and $k_q = 2q$ for each $q \in [m]$. Suppose $A$ is such that the event $\bigcap_{q=1}^m M(j_q,k_q)$ holds. Then, by Lemma~\ref{structure}, there exists a subcube $C' \in C\<1,\ldots,2m\>$ and a set $A' \subset A$, such that $A'$ spans $C'$, and $A \setminus A' = \{a_1, \ldots, a_m\}$ for some $a_q$ with $\Delta(a_q,C') = \{2q-1,2q\}$ for each $q \in [m]$.

Now we only have to count. We have $2^{2m}$ choices for $C'$ and, given $C'$, $2^{\ell - 2m}$ choices for each $a_q$. The probability that $C'$ is internally spanned is $P(\ell-2m,p)$, and the probability that $A \setminus C' = \{a_1,\ldots,a_m\}$ is exactly $p^m (1 - p)^{2^\ell - 2^{\ell-2m} - m} = p^m \,\E_{\ell,m,p}$.

Finally, we claim that each of these choices gives a different set $A$. Indeed, if $\ell > 2m$ then $\dim(C') \ge 1$, and so $C'$ is the only member of $C\<1,\ldots,2m\>$ with more than one element (the others are either empty, or contain exactly one $a_q$). Thus we can reconstruct $C'$, $A'$ and $a_1,\ldots,a_m$ from $A$. If $\ell = 2m$ then $C'$ is a single vertex, but is the only element of $A$ at distance two from all of the others, since $d(a_q,a_r) \ge 4$ for each $q,r \in [m]$. Note that here we needed $m = \ell/2 \ge 2$.
\end{proof}

We can now prove Lemma~\ref{R2ell}.

\begin{proof}[Proof of Lemma~\ref{R2ell}]
First consider the even case. By Observation~\ref{noofpairs}, there are
$$\frac{1}{m!} \prod_{q=0}^{m-1} {{2\ell - 2q} \choose 2}\; = \; \frac{(2\ell)!}{2^m m!(2\ell-2m)!}$$
ways of choosing $m$ pairwise disjoint pairs $\{j_1,k_1\},\ldots,\{j_m,k_m\} \subset [2\ell]$ with $j_i < k_i$ for each $i$. Observe also that
$$\ds\frac{2^m (2\ell)!}{m! (2\ell - 2m)!} \left( 2^{2\ell - 2m} \right)^m \; = \; a_m \ds\frac{z(2\ell)}{z(2\ell - 2m)}.$$
Thus, by the inclusion-exclusion formula and Lemma~\ref{Mjk}, we have
\begin{eqnarray*}
R(2\ell,p) & = & \Pr_p\left( \bigcup_{j<k} M_{2\ell}(j,k) \right) \; = \; \sum_{m=1}^\ell (-1)^{m+1} \sum_{\{j_1,k_1\},\ldots,\{j_m,k_m\}} \Pr_p\left( \bigcap_{q=1}^{m} M_{2\ell}(j_q,k_q) \right)\\
& = &  \sum_{m \ge 1} (-1)^{m+1} \frac{(2\ell)!}{2^m m!(2\ell-2m)!} \bigg[ 2^{2m} \left( 2^{2\ell - 2m} \right)^m p^m \E_{2\ell,m,p} P(2\ell - 2m,p) \bigg]\\
& = &  \sum_{m \ge 1} (-1)^{m+1} a_m \frac{z(2\ell)}{z(2\ell-2m)} p^m \E_{2\ell,m,p} P(2\ell - 2m,p)
\end{eqnarray*}
as claimed. The odd case is easier. Note first that, by the definitions, we have
$$\Pr_p\big( L_{2\ell+1}(i) \big) \; \le \; 2^{2\ell+1} p P(2\ell,p)$$
for each $i \in [2\ell+1]$, and
$$\Pr_p\big( M_{2\ell+1}(j,k) \big) \; \le \; 2^{2\ell+1} p P(2\ell-1,p)$$
for each $j,k \in [2\ell+1]$. Thus, by the union bound,
\begin{eqnarray*}
R(2\ell+1,p) & \le & \sum_i \Pr_p\big( L_{2\ell+1}(i) \big) \,+\, \sum_{j < k} \Pr_p\big( M_{2\ell+1}(j,k) \big)  \\
& \le &(2\ell+1)2^{2\ell+1} \,p \,P(2\ell,p) \,+\, {{2\ell+1} \choose 2}  2^{2\ell+1} \, p \, P(2\ell - 1, p),
\end{eqnarray*}
as required.
\end{proof}

\subsection{A lower bound for $R(2\ell+1,p)$}

In this section we shall again use inclusion-exclusion, together with some careful counting, in order to prove a lower bound on $R(2\ell+1,p)$ which is sharp up to a constant multiplicative factor. This lower bound will play a crucial role in Section~\ref{PRsec}, below. We remark that in order to prove the proposition we need the full strength of Theorem~\ref{Q} -- if our bounds on $Q(2\ell,p)$ were a bit weaker (say, if the ratio between the upper and lower bounds were $5/2$) then the proof would not work. Since we need this lower bound on $R(2\ell+1,p)$ in order to prove Lemma~\ref{PandR}, below, in some sense the entire proof of Theorem~\ref{2^ell} rests on the following calculation.

\begin{prop}\label{oddlower}
Let $\delta > 0$ be sufficiently small, and let $\ell \in \N$ and $p > 0$ be such that $\ell^2 2^{2\ell} p \le \delta$. Then
$$R(2\ell + 1,p) \; \ge \; \frac{1}{100} \,(2\ell+1)! \,\lambda^{-\ell} \,2^{(\ell+1)^2} p^{\ell+2}.$$
\end{prop}

\begin{proof}
Recall that, by Theorem~\ref{Q}, we have
$$\frac{2}{5} (2\ell)! \, \lambda^{-\ell} \, 2^{\ell^2} p^{\ell+1} \; \le \; Q(2\ell,p) \; \le \; \frac{3}{5} \, (2\ell)! \, \lambda^{-\ell} \, 2^{\ell^2} p^{\ell+1},$$
since $2^{2\ell}p \le \delta$ is sufficiently small. Suppose first that $\ell \le 19$. Then we need to consider only the event $L(1)$, which has probability at least
$$2^{2\ell+1} p (1-p)^{2^{2\ell}} Q(2\ell,p) \; \ge \; \left( \frac{2}{5} - \delta \right) (2\ell)! \,\lambda^{-\ell} \,2^{(\ell+1)^2} p^{\ell+2},$$
and $\frac{2}{5(2\ell+1)} > \frac{1}{100}$, as required.

So suppose now that $\ell \ge 20$. We consider those ways of spanning $Q = [2]^{2\ell+1}$ which use exactly $\ell + 2$ infected sites, and which \emph{sequentially} span a $2\ell$-dimensional subcube of $[2]^{2\ell+1}$. Indeed, define the event
\begin{eqnarray*}
L'(i) \; = \; L'_\ell(i) & = & \textup{the event that }\{i\}\textup{ is a final element for $A$ in $[2]^\ell$},\\
&& \textup{and $Q\<i\>$ is sequentially internally spanned by $A$},
\end{eqnarray*}
and note that $L'(i) \Rightarrow L(i)$. Thus, by inclusion-exclusion,
\begin{eqnarray}
R(2\ell + 1,p) & \ge & \Pr_p\left( \bigcup_{i} L'_{2\ell+1}(i) \right) \nonumber \\
& \ge & \sum_{i} \Pr_p\big(L'_{2\ell+1}(i)\big) \: - \sum_{j < k} \Pr_p\big( L'_{2\ell+1}(j) \cap L'_{2\ell+1}(k) \big) \nonumber \\
& \ge & (2\ell + 1) 2^{2\ell + 1} p (1-p)^{2^{2\ell}} Q(2\ell,p) \; - \; \sum_{j < k} \Pr_p\big( L'_{2\ell+1}(j) \cap L'_{2\ell+1}(k) \big) \nonumber \\
& \ge & \left( \frac{2}{5} - \delta \right) (2\ell + 1)! \, \lambda^{-\ell}\, 2^{(\ell + 1)^2} p^{\ell+2} \; - \; \sum_{j < k} \Pr_p\big( L'_{2\ell+1}(j) \cap L'_{2\ell+1}(k) \big). \label{eq2}
\end{eqnarray}

By symmetry, we simply need to bound the probability of the event $L'_{2\ell+1}(1) \cap L'_{2\ell+1}(2)$ from above. This is harder than it looks, however, since there are many ways in which such an event could occur. Let $b_1$ and $b_2$ denote the unique `last' vertices in directions 1 and 2 respectively. Then $A \setminus \{b_1\}$ sequentially spans $Q\<1\>$, so choose a spanning sequence $(a_0,\ldots,a_\ell)$ for $Q\<1\>$ such that $b_2$ occurs as late as possible, i.e., let $b_2 = a_q$ and let $q$ be maximal. Note that $q \ge 1$, since the first two terms of the sequence are interchangeable.

Let $C = [\{a_0,\ldots,a_{q-1}\}]$, and note that $C$ is sequentially internally spanned by $A$. Let $\dim(C) = 2\ell - 2m$ (so $q + m = \ell + 1$), and let $X = (2\ell + 1)! \,\lambda^{-\ell} \,2^{(\ell + 1)^2} p^{\ell+2}$. We divide into cases as follows:

\bigskip
\noindent Case 1: $\dim(C) = 2\ell - 2$.\\

Without loss of generality, let $C = (0,0,0,*,\dots,*)$. Then $b_1 \in (1,0,1,*,\dots,*)$ and $b_2 \in (0,1,1,*,\dots,*)$, since $Q\<1\>$ and $Q\<2\>$ are sequentially internally spanned. We have ${{2\ell+1} \choose 2}$ choices (in~\eqref{eq2}) for the pair $\{j,k\} = \{1,2\}$, $2^3(2\ell - 1)$ choices for the cube $C$ (given $j$ and $k$), and $(2^{2\ell-2})^2$ choices for the vertices $b_1$ and $b_2$. Thus the probability of this case is at most
\begin{eqnarray*}
{{2\ell+1} \choose 2} 2^3(2\ell - 1)(2^{2\ell-2})^2 p^2 Q(2\ell-2,p) & \le & \frac{3}{20}(2\ell + 1)! \,\lambda^{-\ell + 1} \,2^{(\ell - 1)^2 + 4\ell} p^{\ell+2} \; = \; \frac{3\lambda}{20} X.
\end{eqnarray*}

\bigskip
\noindent Case 2: $\dim(C) \le 2\ell - 4$ and $d(C,b_1) = d(C,b_2) = 2$.\\

Without loss of generality, let $C = (0,\dots,0,*,\dots,*)$, and note that there are $2m+1$ zeros. Since $d(C,b_1) = d(C,b_2) = 2$, and $i \in \Delta(b_i,C)$, we have at most $(2m-1)2^{2\ell-2m}$ choices for each vertex $b_j$; without loss let $b_2 = a_q \in (0,1,1,0\ldots,0,*\ldots,*)$. Observe also that, since $(a_0,\ldots,a_\ell)$ is a spanning sequence for $Q\<1\>$, and
$$\left[ \big\{a_0,\ldots,a_q\big\} \right] \; = \; \big( 0,*,*,0,\ldots,0,*,\ldots,* \big),$$
it follows that $(\textbf{0},a_{q+1},\ldots,a_\ell)$ is a spanning sequence for $Q[4,\ldots,2m+1] \cong [2]^{2m-2}$.

We have ${{2\ell+1} \choose 2}$ choices for the pair $\{j,k\} = \{1,2\}$, $2^{2m+1}{{2\ell - 1} \choose {2m-1}}$ choices for the cube $C$ (given $j$ and $k$), and at most
$$(2m-1)^2(2^{2\ell-2m})^{m+1} 2^{m-1} \bigg( \frac{m}{2^{2m-2}} |\S(m - 1)| \bigg)$$
choices for the vertices $a_{q+1},\ldots,a_\ell$, $b_1$ and $b_2$, given $C$. [Here $\S(\ell)$ is the function in Theorem~\ref{Q}. The factor $m/2^{2m-2}$ is due to the fact that one of the $m$ elements of the set in $\S(m-1)$ is $\textbf{0} = (0,\ldots,0)$. The factor $2^{m-1}$ is due to the fact that we know the first two co-ordinates of the points $a_{q+1},\ldots,a_\ell$ are 0, but we have no control over their third co-ordinates.]

Thus the probability in this case is at most
$${{2\ell+1} \choose 2} 2^{2m+1} {{2\ell - 1} \choose {2m-1}} (2m-1)^2 (2^{2\ell-2m})^{m+1} \left( \frac{m}{2^{m-1}} |\S(m - 1)| \right) p^{m+1} Q(2\ell-2m,p),$$
which is at most
\begin{align*}
& \frac{(2\ell + 1)!}{(2\ell-2m)! \,(2m-1)!} \,m(2m-1)^2 2^{(2\ell-2m+1)(m+1)} \left( \frac{3}{5} (2m-2)! \, \lambda^{-m + 1} 2^{(m-1)^2} \right)\\
& \hspace{9cm} \times \; \left( \frac{3}{5} (2\ell - 2m)! \, \lambda^{-\ell+m} \, 2^{(\ell - m)^2} \right) p^{\ell+2}\\
& \hspace{1cm} \le \; \frac{9\lambda}{25} \Big( m(2m-1) \Big) 2^{-3m+1} \bigg[ (2\ell + 1)! \,\lambda^{-\ell} \,2^{(\ell + 1)^2} p^{\ell+2} \bigg] \; = \; \frac{18 \lambda m(2m-1)}{25 \cdot 2^{3m}} X.
\end{align*}

\bigskip
\noindent Case 3: $\dim(C) \le 2\ell - 4$ and $d(C,b_1) \ge 3$.\\

The calculation is the same as in Case~2, except we must multiply the final probability by $\frac{m(m-1)}{2m-1}2^{- 2\ell + 4m - 2}$. Indeed, let $C = (0,\dots,0,*,\dots,*)$, note that $d(C,b_2) = 2$ (since $b_2 = a_q$), so we have $2m-1$ choices for the first $2m+1$ coordinates of the vertex $b_2$, and without loss of generality let $b_2 \in (0,1,1,0\ldots,0,*\ldots,*)$. We again have that $(\textbf{0},a_{q+1},\ldots,a_\ell)$ is a spanning sequence for $Q[4,\ldots,2m+1] \cong [2]^{2m-2}$.

Now, $A \setminus \{b_2\}$ internally spans $Q\<2\>$, and $d(a_t,C) \ge 3$ for every $q+1 \le t \le \ell$, since $C$ was chosen to be maximal. Thus there must be a (non-trivial) cube, internally spanned by $\{b_1,a_{q+1},\ldots,a_\ell\}$, at distance two from $C$. In particular, there must be two vertices of $\{b_1,a_{q+1},\ldots,a_\ell\}$ which agree on co-ordinates $2m+2,\ldots,2\ell+1$.

Hence the the probability in this case is at most the probability in Case~2 times $\ds\frac{2^{2m-1}}{2m-1}$ (due to our extra choices for the first $2m+1$ co-ordinates of the vertex $b_1$), times ${m \choose 2}$ (for the choice of the two vertices from $\{b_1,a_{q+1},\ldots,a_\ell\}$ which agree on their last $2\ell - 2m$ coordinates), times $2^{-(2\ell-2m)}$ (due to our reduced number of choices on these coordinates). This gives that the probability in this case is at most
$${m \choose 2} \left( \frac{2^{2m-1}}{2m-1} \right) 2^{-(2\ell-2m)} \left( \frac{18 \lambda m(2m-1)}{25 \cdot 2^{3m}} \right) X \; \le \; \frac{m^3}{2^{2\ell - m + 2}} X.$$

\medskip
Finally, we sum over the three cases, and over $m$, which gives
\begin{eqnarray*}
\sum_{j < k} \Pr_p\big( L'_{2\ell+1}(j) \cap L'_{2\ell+1}(k) \big) & \le & \left( \frac{3\lambda}{20} \, + \, \sum_{m = 2}^\ell \frac{18 \lambda m(2m-1)}{25 \cdot 2^{3m}} \,+\, \sum_{m=2}^\ell \frac{m^3}{2^{2\ell - m + 2}} \right) X\\
& \le & \left( \frac{9}{50} \,+\, \frac{18\lambda}{25} \cdot \frac{5}{2^5} \,+\, \frac{\ell^3}{2^{\ell+1}} \right) X \; < \; \frac{19}{50} X,
\end{eqnarray*}
since $\lambda < 6/5$ and $\ell \ge 20$. It follows from \eqref{eq2} that
$$R(2\ell + 1,p) \; \ge \; \left( \frac{2}{5} \,-\, \frac{19}{50} \,-\, \delta \right) X \; > \; \frac{1}{100} (2\ell + 1)! \,\lambda^{-\ell} \,2^{(\ell + 1)^2} p^{\ell+2},$$
if $\delta > 0$ is sufficiently small, as required.
\end{proof}

\subsection{Relating $P(\ell,p)$ and $R(\ell,p)$}\label{PRsec}

Recall that $R(\ell,p) \le P(\ell,p)$ for every $\ell \in \N$. The following lemma, which will take some effort to prove, shows that this approximation is fairly tight. Combining Lemmas~\ref{R2ell} and~\ref{PandR} will allow us to use the technical lemma (Lemma~\ref{tech}), and hence to prove Theorem~\ref{induction}.

\begin{lemma}\label{PandR}
Let $\delta > 0$ be sufficiently small, and let $\ell \in \N$ and $p > 0$ be such that $\ell^2 2^{2\ell} p \le \delta$.  Suppose the properties $(*)$ for $2\ell - 1$ hold. Then
$$P(2\ell,p) \; \le \; \left( 1 + O\left( \frac{\delta}{\ell^2} \right) + \frac{1}{2^\ell} \right) R(2\ell,p).$$
Suppose the properties $(*)$ for $2\ell$ hold. Then
$$P(2\ell+1,p) \; \le \; \frac{5}{4} R(2\ell+1,p).$$
\end{lemma}

\begin{rmk}
We shall in fact prove an upper bound for $P(2\ell+1,p)$ of the form
$$\left( 1 + O\left( \frac{\delta}{\ell^2} + \frac{1}{2^\ell} \right) \right) R(2\ell+1,p).$$
However, the result stated above is actually what we shall need in order to prove Theorem~\ref{induction}.
\end{rmk}

The proof of Lemma~\ref{PandR} requires several preliminary lemmas, as well as some lengthy (but straightforward) calculations, see~\cite{arXiv} for the details. We shall begin with a brief sketch to provide some motivation. Suppose $A$ percolates in $Q = [2]^\ell$; by Lemma~\ref{ST} there exists a pair $(S,T)$ of (proper) subcubes of $Q$ such that $S$ and $T$ are disjointly internally spanned by $A$, and $S$ and $T$ together span $Q$. Moreover we may take $S$ to be an internally spanned proper subcube of $Q$ of maximal size.

We break into five cases:
\begin{itemize}
\item $\dim(S) = \ell - 1$ and $|A \setminus S| = 1$,\smallskip
\item $\dim(S) = \ell - 1$ and $|A \setminus S| \ge 2$,\smallskip
\item $\dim(S) = \ell - 2$ and $|A \setminus S| = 1$,\smallskip
\item $\dim(S) = \ell - 2$ and $|A \setminus S| \ge 2$,\smallskip
\item $\dim(S) \le \ell - 4$.\smallskip
\end{itemize}
Note that $\dim(S) = \ell - 3$ is impossible, since if $d(S,A \setminus S) \le 2$, then $S$ is not maximal. But $T = \{x \in Q : d(S,x) \ge 3\}$ is a subcube, so $S \cup T$ does not percolate.

Observe that in the first case one of the events $L(i)$ occurs, and that in the third case one of the events $M(j,k)$ holds.

\begin{defn}
Given $\ell \in \N$ and $p \in [0,1]$, and a random set $A \sim \Bin([2]^\ell,p)$, we define:
\begin{itemize}
\item[$(a)$] $X(\ell,p)$ to be the probability that there is a subcube $S \subset [2]^\ell$, with $\dim(S) \in \{\ell - 2,\ell-1\}$, such that $S$ is internally spanned by $A$, and $|A \setminus S| \ge 2$.
\item[$(b)$] $Y(\ell,p)$ to be the probability that $[2]^\ell$ is internally spanned by $A$, in such a way that no dimension $\ell - 1$ or $\ell - 2$ cube is internally spanned.
\item[$(c)$] $Z(\ell,p)$ to be the probability that some $(\ell - 1)$-dimensional subcube $S \subset [2]^\ell$ is internally spanned by $A$, and $A \setminus S$ is non-empty.
\end{itemize}
\end{defn}

The following lemma, which gives us our basic upper bound on $P(\ell,p)$, follows easily from the comments above.

\begin{lemma}\label{PRXYZ}
Let $\ell \in \N$ and $p \in [0,1]$. Then
$$P(2\ell,p) \; \le \; R(2\ell,p) \: + \: X(2\ell,p) \: + \: Y(2\ell,p) \: + \: Z(2\ell,p),$$
and
$$P(2\ell+1,p) \; \le \; R(2\ell+1,p) \: + \: X(2\ell+1,p) \: + \: Y(2\ell+1,p).$$
\end{lemma}

\begin{proof}
Suppose that $[2]^{\ell}$ is internally spanned, and apply Lemma~\ref{ST}. If none of the events associated with $X$, $Y$ and $Z$ occur, then we may choose $S$ so that $\dim(S) = \ell - 2$ and $|A \setminus S| = 1$, and thus one of the events $M(j,k)$ holds.

Moreover, if $\ell$ is odd and the event associated with $Z$ occurs, but that associated with $X$ does not, then one of the events $L(i)$ occurs.
\end{proof}

The rest of this section will be devoted to proving upper bounds on $X(\ell,p)$, $Y(\ell,p)$ and $Z(2\ell,p)$, using the induction hypothesis (i.e., the properties $(*)$).

\begin{lemma}\label{XandZ}
Let $\delta > 0$, and let $\ell \in \N$ and $p > 0$ be such that $\ell^2 2^\ell p \le \delta$.  Suppose the properties $(*)$ for $\ell-1$ hold. Then
$$X(\ell,p) \, = \, O\left( \frac{\delta}{\ell^2} \right) R(\ell,p).$$
If moreover $\ell$ is even, then
$$Z(\ell,p) \; = \; O\left( \frac{\delta}{\ell^2} \right) R(\ell,p).$$
\end{lemma}

\begin{proof}
We shall first bound $X(\ell,p)$. There are four cases to consider: $\ell$ is odd or even, and $\dim(S) = \ell - 1$ or $\ell - 2$. We shall give the details here only the case where $\ell = 2t + 1$ is odd, and $\dim(S) = \ell - 1$; the details of the other cases are similar, and may be found in~\cite{arXiv}.

Suppose some $2t$-dimensional subcube $S$ of $[2]^{2t+1}$ is internally spanned, and that $|A \setminus S| \ge 2$. We have $2(2t+1)$ ways of choosing $S$, and the expected number of pairs in $A \setminus S$ is at most ${{2^{2t}} \choose 2}p^2$. Thus, writing $X_1(2t+1,p)$ for the probability that such a subcube $S$ exists, we have
\begin{eqnarray*}
X_1(2t+1,p) & \le & 2(2t+1) {{2^{2t}} \choose 2} p^2 P(2t,p) \; \le \; (2t+1)! \, \lambda^{-t} \,2^{t^2 + 4t} p^{t+3}
\end{eqnarray*}
by the property $(*)(a)$ for $\ell - 1$. Recalling that, by Proposition~\ref{oddlower}
$$R(2t+1,p) \; \ge \; \frac{1}{100} \,(2t+1)! \,\lambda^{-t} \,2^{(t+1)^2} p^{t+2},$$
we obtain
\begin{eqnarray*}
X_1(2t+1,p) & \le & 100 \left( 2^{2t-1} p \right) R(2t+1,p) \; \le \; \frac{25\delta}{\ell^2} R(\ell,p),
\end{eqnarray*}
since $\ell^2 2^{\ell} p \le \delta$.

Now assume that $\ell = 2t$ is even. In order to bound $Z(2t,p)$, observe that there are $2t$ choices for the final direction, and $2^{2t}$ choices for the position of an extra active vertex. Thus, by property $(*)(b)$,
\begin{eqnarray*}
Z(2t,p) & \le & \sum_i \Pr_p(L(i)) \; \le \; 2t \big( 2^{2t}p \big) P(2t-1,p)\\
& \le & 5 \cdot (2t)! \, \lambda^{-t+1} \, 2^{t^2 + 2t} p^{t+2} \; \le \; \frac{25\lambda}{2} (2^{2t} p) Q(2t,p) \; \le \; \frac{15\delta}{\ell^2} R(\ell,p)
\end{eqnarray*}
since $\ell^2 2^{\ell} p \le \delta$ and $\lambda < 6/5$, as required.
\end{proof}

Proving the following bounds on $Y(\ell,p)$ will require more effort.

\begin{lemma}\label{Yeven}
Let $p > 0$ be sufficiently small, and let $\ell \in \N$, with $\ell^2 2^{2\ell}p \le 1$. Suppose the properties $(*)$ for $2\ell - 1$ hold. Then
$$Y(2\ell,p) \; \le \; \frac{1}{2^\ell} \, R(2\ell,p).$$
\end{lemma}

\begin{lemma}\label{Yodd}
Let $p > 0$ be sufficiently small, and let $\ell \in \N$, with $\ell^2 2^{2\ell}p \le 1$. Suppose the properties $(*)$ for $2\ell$ hold.
Then
$$Y(2\ell+1,p) \; \le \; \frac{1}{5} \,R(2\ell+1,p).$$
\end{lemma}

Observe also that $Y(\ell,p) = 0$ if $\ell \le 5$, since if no dimension $\ell - 1$ or $\ell - 2$ cube is internally spanned, then $\dim(S) \le \ell - 4$ for any internally spanned proper subcube $S$, and if $\ell \le 5$ then two cubes of dimension at most $\ell - 4$ cannot span $Q \cong [2]^\ell$.

We shall first bound $Y(\ell,p)$ for large $\ell$, using the properties $(*)$, and then we shall count very carefully, using recurrence relations and a computer program, in the cases where $\ell \le 50$. We make no attempt to optimize the constants in the following lemma, since bounding $Y(\ell,p)$ for $\ell \le 50$ requires little more effort than bounding it for $\ell \le 30$ (see Lemmas~\ref{Yevensmall} and~\ref{Yoddsmall}, below).

\begin{lemma}\label{Ybig}
Let $p > 0$ be sufficiently small, and let $\ell \in \N$ with $\ell^2 2^{\ell}p \le 10$. Suppose the properties $(*)$ for $\ell - 1$ hold. Then
$$Y(\ell,p) \; = \; O\left( \frac{\ell^2}{2^\ell} \right) R(\ell,p).$$
Moreover, when $\ell$ is even the constant implicit in the $O(.)$ term is at most $2^{10}$, and when $\ell$ is odd it is at most $2^{20}$.
\end{lemma}

\begin{proof}
The proof is quite simple, but involves some lengthy calculations, see~\cite{arXiv} for the full details. Suppose that $[2]^\ell$ is internally spanned by $A$; by Lemma~\ref{ST} there exist disjointly internally spanned subcubes $S$ and $T$ as described in that lemma. We show that, if $\dim(S) \le \ell - 4$, then the expected number of such pairs is small.

Let $2 \le m \le k \le \ell - 4$, and suppose that $\dim(S) = k$ and $\dim(T) = m$. There are $2^{\ell - k}{\ell \choose k}$ ways to choose $S$, and each is internally spanned with probability $P(k,p)$. Also, given $S$, there are at most $2^{\ell - m}{\ell \choose m}$ ways to choose $T$, and each is internally spanned with probability $P(m,p)$. Thus, by the van den Berg-Kesten Lemma,
\begin{eqnarray*}
Y(\ell,p) & \le & \sum_{\substack{2 \,\le\, m \,\le\, k \,\le\, \ell - 4 \\[+0.5ex] k + m \,\ge\, \ell - 2}} 2^{2\ell - k - m} {{\ell} \choose k}{{\ell} \choose m} P(k,p) P(m,p).
\end{eqnarray*}
since the events `$S$ is internally spanned' and `$T$ is internally spanned' occur disjointly.

Now, if $\ell$ is even then the properties $(*)$ give upper bounds on $P(k,p)$ and $P(m,p)$, and Theorem~\ref{Q} gives a lower bound on $R(\ell,p)$. The result now follows from a straightforward maximization argument.

If $\ell$ is odd, we split into the cases $k + m \ge \ell-1$ and $k + m = \ell - 2$, and note that in the latter case, at least one of $k$ and $m$ must be odd. Both cases now follow as before, using Proposition~\ref{oddlower} to obtain a lower bound on $R(\ell,p)$.
\end{proof}

This bound is sufficient to prove Lemmas~\ref{Yeven} and~\ref{Yodd} when $\ell$ is large: we need $2^{\ell/2} \ge 2^{10} \ell^2$ when $\ell$ is even, and $2^\ell \ge 5 \cdot 2^{20} \ell^2$ when $\ell$ is odd, so in fact $\ell \ge 50$ suffices.

Finally we deal with the small cases, which must be calculated more carefully. The idea is that, if $p > 0$ is sufficiently small and $\ell \le 50$, then the contribution to $P(\ell,p)$ of those configurations with more than $\lceil \ell/2 \rceil +1$ active sites is at most
$${{2^\ell} \choose {\lceil \ell/2 \rceil + 2}} p^{\lceil \ell/2 \rceil + 2} \; \le \; O(p) R(\ell,p).$$
This restricts the possibilities, and thus allows us to count the remaining configurations more accurately.

Given $\ell \in \N$ and $Q \cong [2]^\ell$, we shall write:
\begin{itemize}
\item[$(a)$] $P^*(\ell)$ for the number of sets $A \subset Q$ of size $\lceil \ell/2 \rceil +1$ which internally span $Q$.\smallskip
\item[$(b)$] $R^*(\ell)$ for the number of such sets for which (if $\ell$ is even) one of the events $M(j,k)$ occurs, or (if $\ell$ is odd) one of the events $M(j,k)$ or $L(i)$ occurs.\smallskip
\item[$(c)$] $Y^*(\ell)$ for the number of sets $A \subset Q$ of size $\lceil \ell/2 \rceil +1$ which span $Q$, but internally span no subcube of dimension $\ell - 1$ or $\ell - 2$.
\end{itemize}
Note that $P^*(\ell) = R^*(\ell) + Y^*(\ell)$ for every $\ell \in \N$, since the events corresponding to $X(\ell,p)$ and (if $\ell$ is even) $Z(\ell,p)$ require at least $\lceil \ell/2 \rceil + 2$ infected sites.

\begin{lemma}\label{Yevensmall}
Let $\ell \le 25$, and let $p > 0$ be sufficiently small. Then
$$Y(2\ell,p) \; \le \; \ds\frac{1}{2^\ell} \,R(2\ell,p).$$
\end{lemma}

\begin{proof}
Suppose that $A \subset Q = [2]^{2\ell}$ spans $Q$, and that $|A| = \ell + 1$. Apply Lemma~\ref{ST}, as in the previous lemma, to obtain disjointly internally spanned subcubes $S$ and $T$ with  $\dim(S) = k$ and $\dim(T) = m$, where $2 \le m \le k \le 2\ell - 4$. Since $|A| = \ell + 1$, we must have $k + m = 2\ell - 2$, and both $k$ and $m$ must be even. Let $\J_\ell$ denote the set of possible pairs $(k,m)$, that is
$$\J_\ell \; := \; \big\{ (k,m) \,:\, 2 \le m \le k \le 2\ell - 4, k + m = 2\ell - 2, \textup{ and both $k$ and $m$ are even} \big\}.$$
There are at most $2^{2\ell}{{2\ell} \choose k}{{2\ell - k} \choose m}$ ways of choosing $S$ and $T$. Thus
\begin{eqnarray}\label{Yrec}
Y^*(2\ell) & \le & \sum_{(k,m) \in \J_\ell} 2^{2\ell} {{2\ell} \choose k}{{2\ell-k} \choose m} P^*(k) P^*(m).
\end{eqnarray}

Recall that $P^*(2\ell) = R^*(2\ell) + Y^*(2\ell)$, and observe that
$$|\S(\ell)| \; \le \; R^*(2\ell) \; \le \; {{2\ell} \choose 2} 2^{2\ell} P^*(2\ell - 2),$$
where $\S(\ell)$ is the set in Theorem~\ref{Q}, by the union bound over the events $M(j,k) \wedge \{|A| = \ell + 1\}$. Recall that we can calculate $|\S(\ell)|$ exactly using~\eqref{Qrec}.

Using these recursions, it is clear that we can obtain an upper bound on $Y^*(2\ell)$ for every $\ell \in \N$. In order to prove the claimed bounds, we first calculate the base cases slightly more carefully. To be precise, we shall use the following easy observations.

\bigskip
\noindent \ul{Claim}: $P^*(4) = R^*(4) = 144$ and $Y^*(6) \le 6! \cdot 2^4$.

\begin{proof}[Proof of claim]
The first part follows because $|\S(2)| = 144$ (by~\eqref{Qrec}), and if $A$ internally spans $Q = [2]^4$ and $|A| = 3$, then $A$ sequentially spans $Q$. To bound $Y^*(6)$, we use~\eqref{Yrec} and the following observation: the only pair $(k,m) \in \J_3$ is $(2,2)$, but now the formula above can be improved by a factor of two, since we are double counting (the cubes $S$ and $T$ are indistinguishable). Thus
$$Y^*(6) \; \le \; 2^5 {6 \choose 2}{4 \choose 2} P^*(2)^2 \; \le \; 6! \cdot 2^4,$$
as claimed.
\end{proof}

Now, using a simple computer program and the recurrence relations above, we obtain the following values for $P^*(2\ell)$, $R^*(2\ell)$ and $Y^*(2\ell)$.

\begin{center}
\begin{tabular}{c||c|c|c|c|c|c|c}
$2\ell$ & 2 & 4 & 6 & 8 & 10 & 12 & 14 \\[+0.5ex]  \hline \hline
& & & & & & & \\[-2ex]
$|\S(\ell)| \approx $ \; & \; 2 \; & \; 144 \; & \; 116160 \; & \; $7 \times 10^8$ \; & \; $3 \times 10^{13}$ \; & \; $7 \times 10^{18}$ \; & \; $8 \times 10^{24}$\; \\[+1ex]
$P^*(2\ell) \le$ \; & 2 & 144 & 149760 & $1.2 \times 10^9$ & $6 \times 10^{13}$ & $2 \times 10^{19}$ & $2 \times 10^{25}$  \\[+1ex]
$R^*(2\ell) \le$ \; & 2 & 144 & 138240 & $1.1 \times 10^9$ & $6 \times 10^{13}$ & $2 \times 10^{19}$ & $2 \times 10^{25}$ \\[+1ex]
$Y^*(2\ell) \le$ \; & 0 & 0 & 11520 & $3.1 \times 10^7$ & $4 \times 10^{11}$ & $3 \times 10^{16}$ & $2 \times 10^{22}$ \\[+1ex]
$2^\ell \cdot \frac{Y^*(2\ell)}{|\S(\ell)|} \le$ \; & 0 & 0 & $0.794$ & 0.696 & 0.440 & 0.264 & 0.155 \\
\end{tabular}\\\
\end{center}
\medskip

Continuing in the same way, we observe that $2^\ell Y^*(2\ell) < (4/5)|\S(\ell)|$ for all $0 \le \ell \le 25$. Thus
$$Y(2\ell,p) \; \le \; Y^*(2\ell)p^{\ell+1} \,+\, {{2^{2\ell}} \choose {\ell + 2}} p^{\ell + 2} \; < \; \left( \frac{4}{5} \,+\, O(p) \right) \frac{1}{2^\ell} R(2\ell,p),$$
so the lemma follows.
\end{proof}

Finally, we need to bound $Y(2\ell+1,p)$ for small $\ell$.

\begin{lemma}\label{Yoddsmall}
Let $p > 0$ be sufficiently small, and let $\ell \le 25$. Then
$$Y(2\ell+1,p) \; \le \; \frac{1}{5} \,R(2\ell+1,p).$$
\end{lemma}

\begin{proof}
The proof is very similar to that of Lemma~\ref{Yevensmall}, but we have to consider two additional cases: the case where $k$ and $m$ have the same parity, and the case where one of them is even and the other odd. Recall the definitions of $Y^*(\ell)$, $P^*(\ell)$ and $R^*(\ell)$ from above, and apply Lemma~\ref{ST}, to obtain disjointly internally spanned cubes $S$ and $T$ with $[S \cup T] = [2]^{2\ell+1}$. Let $k = \dim(S)$ and $m = \dim(T)$, where $k \ge m$.

We first claim that either $k + m = 2\ell - 1$, or $k + m = 2\ell$ and $k$ and $m$ are both even. This follows by Lemma~\ref{minl}, since a witness set for $S$ to be internally spanned has size at least $k/2 + 1$, and one for $T$ has size at least $m/2 + 1$, and $|A| = \ell + 2$. We deal with these two cases separately. In the first case we get
\begin{eqnarray}\label{Y1}
Y^*_1(2\ell+1) & \le & \sum_{\substack{2 \,\le\, m \,\le\, k \,\le\, 2\ell-3 \\[+0.5ex] k + m \,=\, 2\ell - 1}} 2^{2\ell+1} {{2\ell+1} \choose k}{{2\ell+1-k} \choose m} P^*(k) P^*(m),
\end{eqnarray}
and in the second case
\begin{eqnarray*}
Y^*_2(2\ell+1) & \le & \sum_{\substack{4 \,\le\, m \,\le\, k \,\le\, 2\ell-4 \\[+0.5ex] k + m \,=\, 2\ell \\[+0.5ex] k,m \textup{ even}}} 2^{2\ell-1} \frac{(2\ell+1)! (km + 4)}{k! \, m!}  P^*(k) P^*(m),
\end{eqnarray*}
since we have
$$2^{2\ell+1}{{2\ell+1} \choose k}{{2\ell + 1 - k} \choose m} \; + \; 2^{2\ell}{{2\ell+1} \choose k}{k \choose 1}{{2\ell + 1 - k} \choose {m-1}} \; = \; 2^{2\ell-1} \frac{(2\ell+1)!}{k! \, m!} (km + 4)$$
ways of choosing $S$ and $T$. By the comments above, $Y^*(2\ell+1) = Y_1^*(2\ell+1) + Y_2^*(2\ell+1)$.

Next, recall that $P^*(2\ell+1) = R^*(2\ell+1) + Y^*(2\ell+1)$, and observe that
$$2^{2\ell+1}|\S(\ell)| \; \le \; R^*(2\ell+1) \; \le \; (2\ell + 1)2^{2\ell+1}P^*(2\ell) \,+\, {{2\ell+1} \choose 2} 2^{2\ell+1} P^*(2\ell - 1).$$
The lower bound follows by considering the event $L(1)$; the upper bound follows by using the union bound over the events $M(j,k)$ and $L(i)$, as in Lemma~\ref{R2ell}.

Using these recursions, we can obtain an upper bound on $Y^*(2\ell+1)$ for every $\ell \in \N$. We first calculate $P^*(3)$ exactly.

\medskip
\noindent \ul{Claim}: $P^*(3) = R^*(3) = 32$.

\begin{proof}[Proof of claim]
We count the number of ways in which three sites can percolate. If two of the sites are in opposite corners, then the triple percolates: there are 24 such configurations. If not, then each pair of points is at distance two (take two points at distance one, and consider the position of the other). There are eight ways to choose the first point, three to choose the second, and two to choose the third, and we get each triple in $3! = 6$ ways; thus there are 8 configurations of this type. It is easy to see that $P^*(3) = R^*(3)$.
\end{proof}

Now, using a simple computer program and the recurrence relations above, we obtain the following values for $P^*(2\ell+1)$, $R^*(2\ell+1)$ and $Y^*(2\ell+1)$.

\medskip
\begin{center}
\begin{tabular}{c||c|c|c|c|c|c|c}
$2\ell+1$ & 1 & 3 & 5 & 7 & 9 & 11 & 13 \\[+0.5ex]  \hline \hline
& & & & & & & \\[-2ex]
$R^*(2\ell+1) \ge$ \; & \; 1 \; & \; 32 \; & \; 4608 \; & \; $1.4 \times 10^7$ \; & \; $3.6 \times 10^{11}$ \; & \; $5 \times 10^{16}$ \; & \; $5 \times 10^{22}$ \; \\[+1ex]
$P^*(2\ell+1) \le$ \; & 1 & 32 & 33280 & $2 \times 10^8$ & $9 \times 10^{12}$ & $2 \times 10^{18}$ & $3 \times 10^{24}$  \\[+1ex]
$R^*(2\ell+1) \le$ \; & 1 & 32 & 33280 & $2 \times 10^8$ & $9 \times 10^{12}$ & $2 \times 10^{18}$ & $3 \times 10^{24}$ \\[+1ex]
$Y^*(2\ell+1) \le$ \; & 0 & 0 & 0 & 1720320 & $6.3 \times 10^{10}$ & $3 \times 10^{15}$ & $7 \times 10^{20}$ \\[+1ex]
$\frac{Y^*(2\ell+1)}{R^*(2\ell+1)} \le$ \; & 0 & 0 & $0$ & 0.116 & 0.171 & 0.047 & 0.013 \\
\end{tabular}\\\
\end{center}
\medskip

Continuing in the same way, we observe that $Y^*(2\ell+1) < (9/50)R^*(2\ell+1)$ for all $0 \le \ell \le 25$. Thus
$$Y(2\ell+1,p) \; \le \; Y^*(2\ell+1)p^{\ell+2} \,+\, {{2^{2\ell+1}} \choose {\ell + 3}} p^{\ell + 3} \; < \; \left( \frac{9}{50} \,+\, O(p) \right) R(2\ell+1,p),$$
and so the lemma follows.
\end{proof}

We can now deduce our upper bounds on $Y(\ell,p)$.

\begin{proof}[Proof of Lemmas~\ref{Yeven} and~\ref{Yodd}]
Let $p > 0$ be sufficiently small, and let $\ell \in \N$, with $\ell^2 2^{2\ell} p \le 1$. Suppose the properties $(*)$ for $2\ell - 1$ hold. If $\ell \le 25$ then we have
$$Y(2\ell,p) \; \le \; \frac{1}{2^\ell} R(2\ell,p),$$
by Lemma~\ref{Yevensmall}, whereas if $\ell > 25$, then
$$Y(2\ell,p) \; \le \; \frac{2^{10} (2\ell)^2}{2^{2\ell}} R(2\ell,p) \; \le \; \frac{1}{2^\ell}  R(2\ell,p),$$
by Lemma~\ref{Ybig}, as required. Now suppose the properties $(*)$ for $2\ell$ hold. If $\ell \le 25$ then
$$Y(2\ell+1,p) \; \le \; \frac{1}{5} R(2\ell+1,p),$$
by Lemma~\ref{Yoddsmall}, and if $\ell > 25$ then
$$Y(2\ell+1,p) \; \le \; \frac{2^{20} (2\ell+1)^2}{2^{2\ell+1}} R(2\ell+1,p) \; \le \; \frac{1}{5} R(2\ell+1,p),$$
by Lemma~\ref{Ybig}, as claimed.
\end{proof}

Finally we may put the pieces together, and prove Lemma~\ref{PandR}.

\begin{proof}[Proof of Lemma~\ref{PandR}]
Let $\delta > 0$ be sufficiently small, and let $p > 0$ and $\ell \in \N$, with $\ell^2 2^{2\ell} p \le \delta$. By Lemma~\ref{PRXYZ},
$$P(2\ell,p) \; \le \; R(2\ell,p) \: + \: X(2\ell,p) \: + \: Y(2\ell,p) \: + \: Z(2\ell,p).$$
Suppose the properties $(*)$ for $2\ell - 1$ hold. Then, by Lemmas~\ref{XandZ} and~\ref{Yeven}, we have
\begin{eqnarray*}
P(2\ell,p) & \le & \left( 1 \, + \, O\left( \frac{\delta}{\ell^2} \right) \, + \, \frac{1}{2^\ell} \right) R(2\ell,p),
\end{eqnarray*}
as required. Next, by Lemma~\ref{PRXYZ} we have
$$P(2\ell+1,p) \; \le \; R(2\ell+1,p) \: + \: X(2\ell+1,p) \: + \: Y(2\ell+1,p).$$
Suppose the properties $(*)$ for $2\ell$ hold. Then, by Lemmas~\ref{XandZ} and \ref{Yodd} we have
$$P(2\ell+1,p) \; \le \; \left( 1 \,+\, O\left( \frac{\delta}{\ell^2} \right) \,+\, \frac{1}{5} \right) R(2\ell+1,p) \; \le \; \frac{5}{4} R(2\ell+1,p),$$
since $\delta > 0$ is chosen to be sufficiently small, as required.
\end{proof}

\subsection{Proof of Theorem~\ref{induction}}

Theorem~\ref{induction} now follows by combining Lemmas~\ref{R2ell} and \ref{PandR}, and using the technical lemma, Lemma~\ref{tech}.

\begin{proof}[Proof of Theorem~\ref{induction}]
Let $C > 0$ be a constant to be chosen later, let $\delta = \delta(C) > 0$ be sufficiently small, and let $p > 0$. For each $\ell \in \N_0$, define a function $f$ by
\begin{equation}\label{fdef0}
R(2\ell,p) \; = \; f(2\ell) \,z(2\ell) \,p^{\ell+1}(1-p)^{2^{2\ell}-\ell-1}
\end{equation}
and
\begin{equation}\label{fdef1}
R(2\ell+1,p) \; = \; f(2\ell+1) (2\ell+1)! \,\lambda^{-\ell} \,2^{(\ell+1)^2} \,p^{\ell+2},
\end{equation}
where $z(2\ell) = (2\ell)! \lambda^{-\ell} 2^{\ell^2}$ is the function defined in~\eqref{zdef0}. Also, for each $\ell \in \N_0$ define
$$P(\ell,p) \; = \; h(\ell)R(\ell,p).$$
We claim that there exists a function $g : \N_0 \to \RR^+$, with $\ds\sum_{m=0}^\infty g(m) \le \ds\frac{1}{3}$, such that the following four conditions hold for every $\ell \in \N$ with $\ell^22^{2\ell}p \le \delta$:
\begin{enumerate}
\item[$(a)$] $f(2\ell) \; \le \; \ds\frac{\lambda}{2} \exp\left( \ds\frac{1}{2- \lambda} \ds\sum_{m=0}^{\ell-1} g(m) \right),$\\[+0.3ex]
\item[$(b)$] $f(2\ell + 1) \; \le \; 4,$\\[+0.3ex]
\item[$(c)$] $1 \, \le \, h(2\ell) \, \le \, 1 + g(\ell)$. \\
\item[$(d)$] $1 \, \le \, h(2\ell+1) \, \le \, \ds\frac{5}{4}$.\\[-2ex]
\end{enumerate}
In fact, we shall prove, by induction on $\ell$, that these conditions hold for the function $g(0) = 0$, $g(1) = g(2) = C\delta$, and
$$g(m) \; = \; \frac{C\delta}{m^2} \,+\, \frac{1}{2^m}$$ for every $m \ge 3$. 

\smallskip
The proof will be by induction on $\ell$, so we begin with the base cases. We first bound $f(\ell)$ for $\ell \le 5$; to be precise we claim that 
$$f(0) = 1, \;\; f(1) = 1/2, \;\; f(2) = \lambda/2, \;\; f(3) < 1, \;\; f(4) < \lambda/2, \;\; f(5) < 1.$$
The values for $\ell \le 2$ follow since $R(0,p) = p$, $R(1,p) = p^2$ and $R(2,p) = 2p^2(1-p)^2$. For $3 \le \ell \le 5$, recall (from the proof of Lemmas~\ref{Yevensmall} and~\ref{Yoddsmall}) that $R(3,p) = 32p^3 + O(p^4)$, $R(4) = 144p^3 + O(p^4)$, and $R(5) \le 33280p^4 + O(p^5)$, and note that $p \le \delta / \ell^{2} 2^{2\ell} \le \delta$ is sufficiently small. Thus
$$f(3) \; \le \; \ds\frac{33p^3}{3! \lambda^{-1} 2^4 p^3} \; < \; 1,$$
and similarly $f(4) < \frac{150}{4! \lambda^{-2} 2^4} < \lambda/2$, and $f(5) < \frac{33300}{5! \lambda^{-2}2^9} < 1$. Next, we observe that $h(0) = h(1) = 1$, by the bounds above, and claim that $$1 \; \le \; h(\ell) \; \le \; 1 \,+\, C\delta$$ 
for $\ell \in \{2,3,4,5\}$, as long as $C$ is chosen to be sufficiently large. Indeed, recall that if $\ell \le 5$ then $Y(\ell,p) = 0$, and thus $P(\ell,p) = R(\ell,p) + O\big(p^{\lceil \ell/2 \rceil + 2}\big)$, by Lemma~\ref{PRXYZ}. But $p \le \delta$, and so the bounds on $h(\ell)$ follow.

So let $\ell \ge 3$, and assume that the claimed upper bounds on $f(t)$ and $h(t)$ hold for every $t \le 2\ell - 1$. It is important to note that this implies that the properties $(*)$ hold for $2\ell - 1$. We shall prove the claimed bounds for $t = 2\ell$ and $t = 2\ell + 1$; we begin by re-writing (the first part of) Lemma~\ref{R2ell} in a more useful form.\\

\noindent \ul{Claim 1}: $f(2\ell) \; = \; \ds\sum_{m=1}^\ell (-1)^{m+1} a_m \, h(2\ell-2m) \,f(2\ell - 2m).$

\begin{proof}[Proof of claim]
By Lemma~\ref{R2ell}, since $\ell \ge 2$,  we have
$$R(2\ell,p) \; = \; \sum_{m \ge 1} (-1)^{m+1} a_m \,p^m \,\E_{2\ell,m,p} \,\frac{z(2\ell)}{z(2\ell-2m)} P(2\ell - 2m,p),$$
and by the definitions above,
$$P(2\ell) \; = \; h(2\ell)R(2\ell,p) \; = \; h(2\ell)\,f(2\ell)\, z(2\ell)\, p^{\ell+1} (1-p)^{2^{2\ell}-\ell-1}.$$
The claim now follows with a little algebra.
\end{proof}

Now, recall that by the induction hypothesis,
$$h(2\ell - 2m) \; \le \; 1 \,+\, g(\ell - m)$$ for each $m \in [\ell]$. Thus we may apply Lemma~\ref{tech}, to obtain
$$f(2\ell) \; \le \; \frac{\lambda}{2} \, \exp\left( \frac{1}{2- \lambda} \ds\sum_{m=0}^{\ell-1} g(m) \right)$$
as required. Also, since the properties $(*)$ for $2\ell-1$ hold, by Lemma~\ref{PandR} we have
$$h(2\ell) \; \le \; 1 \,+\, \frac{C\delta}{\ell^2} \,+\, \frac{1}{2^\ell} \; = \; 1 + g(\ell),$$
since $\ell \ge 3$, as long as we chose $C$ to be sufficiently large.

We have thus proved the claimed upper bounds on $f(t)$ and $h(t)$ for $t = 2\ell$. Observe that therefore the properties $(*)$ hold for $2\ell$.

Next we write the other part of Lemma~\ref{R2ell} in a more useful form.\\

\noindent \ul{Claim 2}: $f(2\ell+1) \, \le \, 1 \,+\, \ds\frac{\lambda}{2} \,h(2\ell-1)\,f(2\ell-1).$

\begin{proof}[Proof of claim]
Recall that, since $\ell \ge 2$, by Lemma~\ref{R2ell} we have
\begin{eqnarray*}
R(2\ell+1,p) & \le & (2\ell+1)\,2^{2\ell+1}\, p \,P(2\ell,p) \,+\, {{2\ell+1} \choose 2}  2^{2\ell+1} \, p \, P(2\ell -1, p).
\end{eqnarray*}
We also have, by the properties $(*)$ for $2\ell$, that $P(2\ell,p) \le (2\ell)! \, \lambda^{-\ell} \, 2^{\ell^2} p^{\ell+1}$, and that
$$P(2\ell-1,p) \; = \; h(2\ell-1)\,f(2\ell-1)\,  (2\ell-1)!\, \lambda^{-\ell+1} \,2^{\ell^2}  p^{\ell+1}.$$
The result now follows with a little algebra.
\end{proof}

Now, by the induction hypothesis, $h(2\ell-1) \le 5/4$ and $f(2\ell - 1) \le 4$. Hence, by Claim~2 it follows that
\begin{eqnarray*}
f(2\ell+1) & \le & 1 \,+\, \ds\frac{\lambda}{2} \,h(2\ell-1)\,f(2\ell-1) \; \le \; 1 \, + \, \frac{3}{4}\, f(2\ell-1) \; \le \; 4,
\end{eqnarray*}
since $\lambda < 6/5$, as required. The bound $h(2\ell+1) \le 5/4$ follows by Lemma~\ref{PandR}, since the properties $(*)$ for $2\ell$ hold.

Hence we have proved the claimed upper bounds on $f(t)$ and $h(t)$ for $t = 2\ell+1$, and the induction step is complete. Taking $\delta = \delta(C)$ sufficiently small, we have
$$\sum_{m=0}^\infty g(m) \; \le \; 2C\delta \,+\, \sum_{m=3}^\infty \left( \frac{C\delta}{m^2} \,+\, \frac{1}{2^m} \right) \; \le \; \frac{1}{4} \,+\, 3C \delta \; \le \; \frac{1}{3},$$
and Theorem~\ref{induction} follows.
\end{proof}

Finally, we can deduce Theorem~\ref{2^ell}.

\begin{proof}[Proof of Theorem~\ref{2^ell}]
The lower bound in part $(a)$ follows from Theorem~\ref{Q}, and the lower bound in part $(b)$ follows by Proposition~\ref{oddlower}. The upper bounds follow easily from Theorem~\ref{induction}, as noted in Section~\ref{Pdefsec}.
\end{proof}

\section{An upper bound for the critical probability}\label{uppersec}

In this section we shall deduce the upper bounds in Theorems~\ref{hypercube} and~\ref{n^d} from Theorem~\ref{Q}. We shall not need Theorem~\ref{2^ell} in order to prove the upper bounds.

Let $n = n(d)$ be a function satisfying $d \gg \log n \ge 1$ as $d \to \infty$. We shall prove that if
$$p \; = \; 4\lambda \left( \frac{n}{n-1} \right)^2 \frac{1}{d^2}\left( 1 \,+\, \frac{5 (\log d)^2 + 11 \log n - 11}{\sqrt{d \log n}} \right) 2^{- 2 \sqrt{d\log n}},$$ and $A \sim \Bin([n]^d,p)$, then $\Pr(A \textup{ percolates}) \to 1$ as $d \to \infty$. Note that when $n(d) = 2$ for all $d \in \N$, this gives the desired bound on the hypercube, and when $d \gg \log n$ it gives the bound in Theorem~\ref{n^d}.

Let $\eps = \eps(d) \in \RR$ with $|\eps(d)| = O(1)$, let $\alpha = \sqrt{d \log n} - \left\lfloor \sqrt{d \log n} \right\rfloor \in [0,1)$, and set
$$\beta = \log \left[ 4\lambda \left( \frac{n}{n-1} \right)^2 \right] + \eps.$$
Set $\ell = \lfloor \sqrt{d \log n} \rfloor + \lfloor \log \sqrt{d} \rfloor$, and note that $\ell \ll d$, since $d \gg \log n$, and that
$$d \log n \; = \; \Big( \ell - \left\lfloor \log \sqrt{d} \right\rfloor + \alpha \Big)^2.$$
Furthermore, set
\begin{equation}\label{sdef1}
s = 2\sqrt{d \log n} + 2 \log d - \beta,
\end{equation}
and let $p = 2^{-s}$. Note that $d 2^{2\ell} p = \Theta(1)$. Write $\X(2\ell,p)$ for the number of $[2]^{2\ell}$-subcubes of $[n]^d$ which are \emph{sequentially} internally spanned by a set $A \sim \Bin([n]^d,p)$.

\begin{lemma}\label{exX}
Let $n = n(d)$ satisfy $d \gg \log n \ge 1$, and let $\ell \in \N$ and $p > 0$ be as defined above. Then, if $d$ is sufficiently large,
$$\Ex\big( \X(2\ell,p) \big) \; \ge \; \frac{2^{\eps \ell}}{n^{10}}.$$
\end{lemma}

\begin{proof}
There are $\ds{d \choose {2\ell}} n^d \left( \ds\frac{n-1}{n} \right)^{2\ell}$ different $[2]^{2\ell}$-subcubes in $[n]^d$, and each is sequentially internally spanned with probability $Q(2\ell,p)$. Observe also that
$$d\log n \,+\, \ell^2 \; = \; \Big( \ell - \lfloor \log \sqrt{d} \rfloor + \alpha \Big)^2 \, + \, \ell^2  \; \ge \; 2\ell^2 \,-\, 2\ell \lfloor \log \sqrt{d} \rfloor \,+\, 2 \alpha \ell \,+\, \frac{1}{5} \left( \log d \right)^2 $$
and that
\begin{eqnarray*}
s(\ell + 1) & = & \bigg( 2\sqrt{d \log n} + 2 \log d - \beta \bigg) \Big( \ell + 1 \Big) \\
& = & \bigg( 2\ell - 2\lfloor \log \sqrt{d} \rfloor + 2\alpha + 2 \log d - \beta \bigg) \Big( \ell + 1 \Big)\\
& \le & 2\ell^2 \,-\, 2\ell \lfloor \log \sqrt{d} \rfloor \,+\, 2 \ell \bigg( \log d + 1 + \alpha - \frac{\beta}{2} \bigg) \, + \, 2\log d.
\end{eqnarray*}
Note that $2^{2\ell}p = o(1)$ as $d \to \infty$. Thus, by Theorem~\ref{Q},
\begin{eqnarray*}
\Ex\big( \X(2\ell,p) \big) & \ge & {d \choose {2\ell}} n^d \left( \frac{n - 1}{n} \right)^{2\ell} \, \left( \frac{2}{5} \, (2\ell)! \, \lambda^{-\ell} \,2^{\ell^2} p^{\ell + 1} \right)\\
& \ge & \frac{2}{5} \left( \frac{d - 2\ell}{\sqrt{\lambda}} \right)^{2\ell} \left( \frac{n - 1}{n} \right)^{2\ell} 2^{d\log n + \ell^2 - s(\ell + 1)}\\
& \ge & \left( \frac{d - 2\ell}{\sqrt{\lambda}} \right)^{2\ell} \left( \frac{n - 1}{n} \right)^{2\ell} 2^{-2\ell(\log d - \beta/2 + 1)}.
\end{eqnarray*}
But
$$\left( \frac{n-1}{n} \right) \left( \frac{1}{\sqrt{\lambda}} \right) 2^{\beta/2 - 1} \; = \; \frac{1}{2} \left( \frac{n-1}{n} \right) \left( \frac{1}{\sqrt{\lambda}} \right) 2^{\eps/2} \left[ 4\lambda \left( \frac{n}{n-1} \right)^2 \right]^{1/2} \; = \; 2^{\eps/2},$$
and
$$\left( \frac{d - 2\ell}{d} \right)^{2\ell} \; \ge \; \exp\left( \frac{-5\ell^2}{d} \right) \; \ge \; \frac{1}{n^{10}},$$
if $d$ is sufficiently large, since $\ell^2 = \big(1 + o(1) \big) d \log_2 n \le 2d \ln n$, and $\ell \ll d$. Thus
\begin{eqnarray*}
\Ex\big( \X(2\ell,p) \big) & \ge & \frac{2^{\eps \ell}}{n^{10}},
\end{eqnarray*}
as required.
\end{proof}

We wish to deduce that, with high probability, some $[2]^{2\ell}$-subcube really is internally spanned; we shall do so by bounding the variance of $\X(2\ell,p)$. In~\cite{BB} it was shown that
$$\Var\big( \X(2\ell,p) \big) \, = \, 2^{o(\ell)}.$$
We shall need the following more precise bound on the variance. (Our bound is in fact close to best possible, see Remark~\ref{rmk} below.) Here we make crucial use of the fact that we are counting sequentially internally spanned subcubes.

\begin{lemma}\label{varX}
Let $p > 0$ and $\ell \in \N$, with $2\ell \le d$ and $2^{2\ell} p \le 1$. Then
$$\Var\big( \X(2\ell,p) \big) \; \le \; d^{3\log d}\,\Ex \big( \X(2\ell,p) \big).$$
\end{lemma}

We shall use the following lemma in the proof of Lemma~\ref{varX}.

\begin{lemma}\label{countAB}
Let $\ell \in \N$ and $0 \le m \le \ell$. Let $B \subset [n]^d$ with $|B| = \ell - m + 1$. Define
$$\S(B,\ell) \; := \; \big\{ A \supset B \,:\, |A| = \ell + 1 \textup{ and $A$ sequentially spans a hypercube } Q \subset [n]^d \big\}.$$
Then
$$|\S(B,\ell)| \; \le \; d^{2m+2} \prod_{j=1}^m 2^{2\ell-2j+3}.$$
\end{lemma}

Note that, if $A \in S(B,\ell)$, then the definition implies that the hypercube $Q = [A]$ has dimension $2\ell$.

\begin{proof}
We cover the set $\S(B,\ell)$ with sets $\S_q$ as follows. For each $0 \le q \le \ell$, let $\S_q$ denote the collection of sets $A \in \S(B,\ell)$ such that there exists a spanning sequence $(a_0,\ldots,a_\ell)$ for $[A]$ (i.e., an ordering of $A$ such that $d(a_{j+1},[\{a_0,\ldots,a_j\}]) = 2$ for each $0 \le j \le \ell - 1$), such that $q = \min\{t : a_t \in B\}$.\\

\noindent \ul{Claim}: $|\S_q| \le d^{2m+1} \prod_{j=1}^m 2^{2\ell-2j+3}$ for each $0 \le q \le \ell$.

\begin{proof}[Proof of Claim]
Let $C = A \setminus B$, and note that $|C| = m$. We are required to count the possible sets $C$ such that $B \cup C \in \S_q$.

First we choose the element $a_q \in B$ (at most $\ell + 1$ choices). Next, if $q \ge 1$, we choose the $2q$-dimensional hypercube $R = [\{a_0,\ldots,a_q\}]$. Note that $a_q \in R$, so there are at most
$$2^{2q} {d \choose {2q}}$$
choices for $R$. Moreover, given $R$, there are at most $(2q)! 2^{q^2}$ choices for the set $C \cap R$, by Theorem~\ref{Q}, since $(C \cap R) \cup \{a_q\}$ must sequentially span $R$. If $q = 0$ then let $R = \{a_0\}$.

Now let $R_1$ denote the largest hypercube in $[R \cup B]$ such that $R \subset R_1$, and let $2r_1 = \dim(R_1)$. We claim that there exists a vertex $c_1 \in C$ with $d(c_1,R_1) = 2$; indeed, the minimal element $c$ of $C \setminus \{a_0,\ldots,a_{q-1}\}$ (in the order $a_0 < \dots < a_\ell$) is at distance exactly two from $R_1$. This follows because $(a_0,\ldots,a_\ell)$ is a spanning sequence for $[A]$ (so $d(c,R_1) \le 2$), and if $d(c,R_1) < 2$ then the dimension of $[R_1 \cup \{c\}]$ would be too small, and $A$ would not span a hypercube of dimension $2\ell$.

There are at most
$${d \choose 2} 2^{2r_1+2}$$
choices for $c_1$, since $c_1 \in R_1 \pm e_i \pm e_j$ for some pair of directions $\{i,j\}$. Now iterate this process: to be precise, for each $2 \le j \le m - q$, let $R_j$ denote the largest hypercube in $[R_{j-1} \cup B]$ with $R \subset R_j$, let $2r_j = \dim(R_j)$, and choose a vertex $c_j \in C$ with $d(c_j,R_j) = 2$. There are at most ${d \choose 2} 2^{2r_j+2}$ choices for $c_j$.

Finally, putting the pieces together, we have at most
$$(\ell + 1) 2^{2q} {d \choose {2q}}(2q)! \,2^{q^2} \max_{\textbf{r}} \prod_{j=1}^{m-q} {d \choose 2} 2^{2r_j+2}$$
choices for $C$, where the maximum is taken over all sequences $\textbf{r} = (r_1,\ldots,r_{m-q})$ with $q \le r_1 < \ldots < r_{m-q} < \ell$. Note we may assume that $2\ell \le d$, since otherwise $|\S(B,\ell)| = 0$. Hence
$$|\S_q| \; \le \; d^{2m+1} 2^{q^2 + 2q} \prod_{j=1}^{m-q} 2^{2\ell-2j+2} \; \le \; d^{2m+1} \prod_{j=1}^m 2^{2\ell-2j+3},$$
as claimed, since $q \le m \le \ell$, and so $\ds\sum_{j = m-q+1}^m \big( 2\ell - 2j + 2 \big) \ge q(q+1)$.
\end{proof}

The lemma follows immediately from the claim, since $\S(B,\ell) \subset \bigcup_q \S_q$.
\end{proof}

\begin{rmk}
In the proof above we used (and proved) the following nice fact about spanning sequences. Suppose that $(a_0,\ldots,a_\ell)$ is a spanning sequence for $Q$, and let $\{b_0,\ldots,b_t\} \subset \{a_0,\ldots,a_\ell\}$ with $b_0 = a_0$ and $d(b_{j+1},[\{b_0,\ldots,b_j\}]) = 2$ for each $j \le t-1$. Then $(b_0,\ldots,b_t)$ can be extended to a spanning sequence of $Q$ that is a permutation of $(a_0,\ldots,a_\ell)$.
\end{rmk}

We can now prove our claimed upper bound on the variance of $\X(2\ell,p)$.

\begin{proof}[Proof of Lemma~\ref{varX}]
We are required to bound $\Ex\big( \X(2\ell,p)^2 \big)$, the expected number of pairs $(Q,Q')$, where $Q$ and $Q'$ are $2\ell$-dimensional subcubes of $[n]^d$, such that both $Q$ and $Q'$ are sequentially internally spanned by $A$. We do so by first bounding, for each $0 \le m \le \ell + 1$, the expected number of pairs for which $|A \cap Q \cap Q'| = \ell + 1 - m$, and then summing over $m$. Recall that, since $Q$ and $Q'$ are sequentially spanned by $A$, it follows that $|A \cap Q| = |A \cap Q'| = \ell + 1$.

First note that if $m = \ell + 1$, then the sets $A \cap Q$ and $A \cap Q'$ are disjoint, and so, by the van den Berg-Kesten Lemma, the expected number of pairs is at most $\Ex\big( \X(2\ell,p)\big)^2$. So fix $0 \le m \le \ell$, fix the cube $Q \subset [n]^d$, and recall that the number of ways of choosing a set $A \cap Q$ which sequentially internally spans $Q$ is denoted $|\S(\ell)|$.

Now, given the set $A \cap Q$ which sequentially spans $Q$, and assuming that $|A \cap Q \cap Q'| = \ell + 1 - m$, we have ${{\ell+1} \choose m} \le d^m$ ways of choosing the set $B = A \cap Q \cap Q'$. By Lemma~\ref{countAB}, we then have at most $d^{2m+2} \prod_{j=1}^m 2^{2\ell-2j+3}$
choices for the set $A \cap Q'$, given $B$. Hence, given $Q$ and $m$, the number of choices for the pair $(A \cap Q, A \cap Q')$ is at most
$$|\S(\ell)| \, d^{3m+2} \prod_{j=1}^m 2^{2\ell-2j+3}.$$
Thus,
\begin{align*}
& \Ex\big( \X(2\ell,p)^2 \big) \; \le \; \Ex\big( \X(2\ell,p)\big)^2 \,+\, \sum_{m=0}^{\ell} \;\sum_{Q :\, \dim(Q) = 2\ell} \, |\S(\ell)| \, d^{3m+2} \,\left[ \prod_{j=1}^{m} 2^{2\ell-2j+3} \right] \, p^{\ell + m + 1}
\end{align*}
and hence, using the bounds $2^{2\ell}p \le 1$ and $2\ell \le d$,
\begin{eqnarray*}
\Ex\big( \X(2\ell,p)^2 \big) \, - \, \Ex\big( \X(2\ell,p)\big)^2 & \le & \sum_{m=0}^{\ell} \Ex\big( \X(2\ell,p) \big) \, (2d)^{3m+2} \,\prod_{j=1}^{m} 2^{-2j} \\
& \le & d^{3\log d}\,\Ex(\X(2\ell,p)),
\end{eqnarray*}
as required. The final inequality follows because the maximum in the sum occurs when $m \approx 3(\log d) / 2$.
\end{proof}

\begin{rmk}\label{rmk}
If $d^{2-c} 2^{2\ell} p \ge 1$ for some $c > 0$, and $\log d \ll \ell \ll d$, then we also have
$$\Var\big( \X(2\ell,p) \big) \; \ge \; d^{\delta \log d}\,\Ex \big( \X(2\ell,p) \big)$$
for some $\delta = \delta(c) > 0$, so the upper bound in Lemma~\ref{varX} is close to best possible. To see this, consider only the cases $m = \ell + 1$ and $m = (c/2) \log d$, and assume that the set $B = A \cap Q \cap Q'$ consists of the first $\ell + 1 - m$ elements of the spanning sequences of both $Q$ and $Q'$. We obtain (approximately, since we did not prove a lower bound in the case $m = \ell + 1$),
\begin{eqnarray*}
\Ex\big( \X(2\ell,p)^2 \big) \,-\, \Ex\big( \X(2\ell,p)\big)^2 & \gtrsim & \Ex\big( \X(2\ell,p) \big) \,\frac{1}{m!} \prod_{j=1}^{m} \left[ {{d - 2\ell} \choose 2} 2^{2\ell-2j} p \right]\\[+0.5ex]
& \ge & \Ex(\X(2\ell,p)) \, \left( \frac{d^{c}}{3m} \right)^m \, \prod_{j=1}^{m} 2^{-2j} \; \gg \; d^{\delta \log d} \, \Ex(\X(2\ell,p)),
\end{eqnarray*}
for any $\delta < c^2/4$, as claimed.

Recall that $d 2^{2\ell} p \approx 1$ for $\ell \approx \sqrt{d \log n} + \log \sqrt{d}$ and $p = 2^{-s}$, as defined above, so this lower bound holds in the case in which we are interested.
\end{rmk}

We can now deduce that $\X(2\ell,p)$ is large with high probability by the standard second moment method (see \cite{RG}, for example).

\begin{lemma}[Chebyshev's inequality]\label{cheby}
Let $X$ be a random variable, and let $t > 0$. Then
$$\Pr\big( |X - \Ex(X)| \ge t  \big) \; \le \; \frac{\Var(X)}{t^2}.$$
and hence
$$\Pr\big( X = 0 \big) \; \le \; \frac{\Var(X)}{\Ex(X)^2}.$$
\end{lemma}

The following result, which follows immediately from Lemmas~\ref{exX}, \ref{varX} and \ref{cheby}, summarises what we have proved so far.

\begin{cor}\label{manyells}
Let $n = n(d)$ satisfy $d \gg \log n \ge 1$, let $\ell = \lfloor \sqrt{d \log n} \rfloor + \lfloor \log \sqrt{d} \rfloor$, and let $p = 2^{-s} > 0$ be as defined in~\eqref{sdef1}. Suppose that $\eps \ell - 3(\log d)^2 - 10 \log n \to \infty$ as $d \to \infty$. Then,
$$\Pr\big( \X(2\ell,p) = 0 \big) \; \to \; 0$$
as $d \to \infty$.
\end{cor}

\begin{proof}
By Lemmas~\ref{cheby}, \ref{varX} and~\ref{exX} respectively, we have
$$\Pr\big( \X(2\ell,p) = 0 \big) \; \le \;  \frac{\Var(\X)}{\Ex(\X)^2} \; \le \; \frac{d^{3\log d}}{\Ex(\X)} \; \le \; \frac{d^{3 \log d}n^{10}}{2^{\eps \ell}} \; \to \; 0$$
as $d \to \infty$, since $\eps \ell - 3(\log d)^2 - 10 \log n \to \infty$.
\end{proof}

Corollary~\ref{manyells} tells us that, with high probability, there exists an internally spanned $[2]^{2\ell}$-subcube in $[n]^d$. All that remains is to show that such a subcube has probability $1 - o(1)$ of infecting the entire vertex set.

Indeed, we want to grow one of our cubes, $Q$, step by step -- first to a $[2]^{d/2}$-cube, and then to an $[n]^d$-cube -- and show that at each step the probability that we get stuck is small. There is a problem however, because the steps are not independent of each other; in fact we are (essentially) summing over a large number of overlapping paths. We solve this problem (as in \cite{BB}) using `sprinkling', i.e., by dividing up $p$ into pieces, and using a different set of sites for each step.

To be precise, we shall set
$$p^* = \ds\frac{1}{d^{5/2}} 2^{-2\sqrt{d \log n}},$$ and for each $j \in \N$, let $p_j = 2^{-j}p^*$ and let $A_j \sim \Bin([n]^d,p_j)$ be independently chosen sets of infected vertices. Observe that
$$2^{2\ell}p^* \; \ge \; \frac{1}{16d^{3/2}},$$
where $\ell = \lfloor \sqrt{d \log n} \rfloor + \lfloor \log \sqrt{d} \rfloor$ is the function defined above.

\begin{lemma}\label{grow2s}
Let $d,\ell \in \N$ and $p^* > 0$ satisfy $2^{2\ell}p^* \ge 1/(16 d^{3/2})$, and let $Q \subset [2]^d$ be a hypercube with $\dim(Q) = 2\ell$. Let $A' = Q \cup A$, where $A \sim \Bin([2]^d,p^*)$. Then,
$$\Pr_{p^*}\Big( \exists \, R \subset [A'] \,:\, \dim(R) \ge d/2 \Big) \; \ge \; 1 \,-\, 2\exp\left( -\frac{1}{2^9} \sqrt{d} \right) \; \to \; 1$$
as $d \to \infty$.
\end{lemma}

\begin{proof}
For each $j \in \N$, set $k = 2\ell + 2j - 2$ and $p_j = 2^{-j}p^*$. Let $\Pr_{p_j}(2^k \not\to 2^{k+2})$ denote the probability that, given a $[2]^k$-cube $Q' \subset [2]^d$ and a random set $A_j \sim \Bin([2]^d,p_j)$, there is no active site $v \in A_j$ satisfying $d(v,Q') = 2$. Note that if $d(v,Q')= 2$ then $[Q' \cup \{v\}]$ is a $[2]^{k+2}$-cube.\\

\noindent \ul{Claim}: If $k = 2\ell + 2j - 2 \le d/2$, then $\Pr_{p_j}(2^k \not\to 2^{k+2}) \le \exp\left( -2^{j-10} \sqrt{d}  \right)$.

\begin{proof}[Proof of claim]
Let $Q'$ be any $k$-dimensional subcube of $[2]^d$. There are at least ${{d/2} \choose 2} \ge d^2/9$ disjoint $[2]^k$-cubes $Q_1,\ldots,Q_t$ in $[2]^d$, such that each vertex $v \in \bigcup_i Q_i$ is at distance exactly 2 from $Q'$.
Thus
\begin{eqnarray*}
\Pr_{p_j}(2^k \not\to 2^{k + 2}) & \le & \big(1 - p_j \big)^{ {{d/2} \choose 2} 2^k} \; \le \; \exp\left( - \left( 2^{-j}p^* \right) \frac{d^2}{9} 2^{2\ell + 2j - 2} \right) \\
& \le & \exp\left( - 2^{j - 10}\sqrt{d} \right),
\end{eqnarray*}
as claimed, since $16 d^{3/2} 2^{2\ell}p^* \ge 1$
\end{proof}

Hence, since $\bigcup_j A_j$ is stochastically dominated by $A \sim \Bin([2]^d,p^*)$, we have
\begin{align*}
& \Pr_{p^*}\Big( \exists \, R \subset [A'] \,:\, \dim(R) \ge d/2 \Big) \; \ge \; 1 \,-\, \sum_{j = 1}^{d/2} \Pr_{p_j}\left( 2^{2\ell + 2j - 2} \not\to 2^{2\ell + 2j} \right) \\
& \hspace{3cm} \ge \; 1 \,-\, \sum_{j = 1}^\infty \exp\left( - 2^{j - 10}\sqrt{d}  \right) \; \ge \; 1 \,-\, 2\exp\left( -\frac{1}{2^9} \sqrt{d} \right) \; \to \; 1,
\end{align*}
as $d \to \infty$, as required.
\end{proof}

Finally we show that each $[2]^{d/2}$-cube almost surely grows to cover the whole of $[n]^d$.

\begin{lemma}\label{growbig}
Let $n = n(d) \in \N$ satisfy $d \gg \log \log n + 1$, and let $p^* = p^*(d) > 0$ satisfy $p^* \ge 2^{-d/3 + o(d)}$ as $d \to \infty$. Let $R \subset [n]^d$ be a hypercube with $\dim(R) = \lfloor d/2 \rfloor$, and let $A' = R \cup A$, where $A \sim \Bin([n]^d,p^*)$. Then
$$\Pr_{p^*}\Big( [A'] = [n]^d \Big) \; \ge \; 1 \,-\, \big( 2n \big)^d \exp\Big( -2^{d/6 + o(d)} \Big) \; \to \; 1$$
as $d \to \infty$.
\end{lemma}

\begin{proof}
If $A'$ does not percolate, then there must exist a $\lfloor d/2 \rfloor$-dimensional hypercube $S \subset [n]^d$ with $A \cap S = \emptyset$. Let $\Q$ denote the collection of such hypercubes, and note that
$$\Pr_{p^*}\big( S \in \Q \big) \; \le \; (1 - p^*)^{|S|} \; \le \; \exp\Big( -p^* 2^{\lfloor d/2 \rfloor} \Big) \; \le \; \exp\Big( -2^{d/6 + o(d)} \Big)$$
for any $S$. There are at most $(2n)^d$ hypercubes $S$ with $\dim(S) = \lfloor d/2 \rfloor$ in $[n]^d$, and thus
\begin{eqnarray*}
\Pr_{p^*}\Big([A'] \neq [n]^d \Big) & \le & \Ex\big( |\Q| \big) \; \le \;  (2n)^d \exp\Big( -2^{d/6 + o(d)} \Big),
\end{eqnarray*}
as required.
\end{proof}

The upper bounds in Theorems~\ref{hypercube} and~\ref{n^d} now follow easily.

\begin{thm}\label{upper}
Let $n = n(d)$ be a function such that $d \gg \log n \ge 1$. Let
$$p \; = \; 4\lambda \left( \frac{n}{n-1} \right)^2  \frac{1}{d^2} \left( 1 \,+\, \frac{5 (\log d)^2 + 11 \log n - 11}{\sqrt{d\log n}} \right) 2^{- 2 \sqrt{d\log n}},$$
and let $A \sim \Bin([n]^d,p)$. Then
$$\Pr_p\big([A] = [n]^d \big) \to 1$$
as $d \to \infty$.
\end{thm}

\begin{proof}
Let
$$\eps \; = \; \eps(d) \; := \; \,\log\left[ 1 \,+ \, \frac{5(\log d)^2 + 11 \log n - \sqrt{\log n}  - 11}{\sqrt{d \log n}} \right],$$
and let $\beta = \log \left[ 4\lambda \left( \ds\frac{n}{n-1} \right)^2 \right] + \eps$. Set $s = 2\sqrt{d \log n} + 2 \log d - \beta$, and note that
$$p \; := \; 4\lambda \left( \frac{n}{n-1} \right)^2  \frac{1}{d^2} \left( 1 \,+\, \frac{5 (\log d)^2 + 11 \log n - 11}{\sqrt{d \log n}} \right) 2^{- 2 \sqrt{d\log n}} \; \ge \; 2^{-s} + 2p^*,$$
where $p^* = d^{-5/2} \, 2^{-2\sqrt{d \log n}}$, as before. Set $p_1 = 2^{-s}$, and let $A^* = A_1 \cup A_2 \cup A_3$, where $A_1 \sim \Bin([n]^d,p_1)$ and $A_2,A_3 \sim \Bin([n]^d,p^*)$. It is easy to see that $A^*$ is stochastically dominated by $A \sim \Bin([n]^d,p)$.

Let $\ell = \lfloor \sqrt{d \log n} \rfloor + \lfloor \log \sqrt{d} \rfloor$, and let $E(1)$ denote the event that some hypercube $Q \subset [n]^d$ with $\dim(Q) = 2\ell$ is sequentially internally spanned by $A_1$. Let $E(2)$ denote the event that some hypercube $R \subset [n]^d$ with $\dim(R) = \lfloor d/2 \rfloor$ is internally spanned by $A_1 \cup A_2$, and let $E(3)$ denote the event that $A^*$ percolates.

Observe that, since $(\log d)^2 + \log n \ll \sqrt{d \log n}$, we have (for large $d$)
$$\eps \ell \; \ge \; \ell \left( \frac{4(\log d)^2 + 10\log n}{\sqrt{d \log n}} \right) \; \ge \; 4(\log d)^2 + 10\log n,$$
and so we may apply Corollary~\ref{manyells}. Hence, by Corollary~\ref{manyells}, Lemma~\ref{grow2s} and Lemma~\ref{growbig}, we have
\begin{eqnarray*}
\Pr_p\Big([A] = [n]^d \Big) & \ge & \Pr\Big( E(1) \Big) \cdot \Pr\Big( E(2) \,\big|\, E(1) \Big) \cdot \Pr\Big( E(3) \,\big|\, E(2) \Big) \; \to \; 1
\end{eqnarray*}
as $d \to \infty$, and the theorem follows.
\end{proof}

\section{A lower bound on the critical probability}\label{lowersec}

In this section we shall deduce the lower bounds in Theorems~\ref{hypercube} and~\ref{n^d} from Theorem~\ref{2^ell}. To be precise, we shall show that, given any function $n = n(d)$ with $d \gg \log n \ge 1$, if
$$p \; = \; 4\lambda \left( \frac{n}{n-1} \right)^2 \frac{1}{d^2}\left( 1 \,+\, \frac{\log d - 16\log n + 16}{\sqrt{d \log n}} \right) 2^{- 2 \sqrt{d\log n}},$$ and $A \sim \Bin([n]^d,p)$, then $\Pr(A \textup{ percolates}) \to 0$ as $d \to \infty$. Note that when $n(d) = 2$ for all $d \in \N$, this gives the desired bound on the hypercube, and when $d \gg \log n$ it gives the bound in Theorem~\ref{n^d}.

We shall use the basic method of~\cite{BB}, but we will need to do some more careful counting, and also to bound the probability that an arbitrary cube (of sufficiently low dimension) is internally spanned. Indeed, let $n = n(d)$ satisfy $d \gg \log n \ge 1$, let $\eps = \eps(d) \in \RR$ with $\eps(d) < 1$, and set
$$\beta = \log \left[ 4\lambda \left( \frac{n}{n-1} \right)^2 \right] \,+\, 2\eps.$$
Set $\ell = 2 \big( \sqrt{d \log n} - \alpha \big) \in \N$, where $\alpha > 0$ is a constant to be chosen later, 
and note that $\ell \ll d$, since $d \gg \log n$, and that $d \log n = (\ell + 2\alpha)^2/4$. Furthermore, set
\begin{equation}\label{sdef}
s = 2\sqrt{d \log n} + 2 \log d - \beta,
\end{equation}
and let $p = 2^{-s}$. Throughout, we assume that $A \sim \Bin([n]^d,p)$. Note that, if $d$ is sufficiently large, then
$$\ell^22^\ell p \; < \; \ds\frac{2^{\beta+2} \log n}{d} \; \le \; \delta,$$
where $\delta > 0$ is the constant in Theorem~\ref{2^ell}, since $d \gg \log n$.

We shall count the expected number of pairs $(S,T)$, such that $S, T \subset [n]^d$ are cubes satisfying the following conditions:
\begin{enumerate}
\item[$(a)$] $S$ and $T$ are disjointly internally spanned by $A$,\\[-2ex]
\item[$(b)$] $\dim(T) \le \dim(S) < \ell$,\\[-2ex]
\item[$(c)$] $U = [S \cup T]$ is a cube, and $\dim(U) \ge \ell$.
\end{enumerate}

\begin{defn}
For each $\ell \in \N$, let $\U_\ell$ denote the collection of pairs $(S,T)$ in $[n]^d$ which satisfy conditions $(b)$ and $(c)$ above. Let $\W_\ell \subset \U_\ell$ denote the number of pairs in $\U_\ell$ which also satisfy condition $(a)$. We shall write $U_\ell = |\U_\ell|$ and $W_\ell = |\W_\ell|$.
\end{defn}

The idea of the proof comes from the following slight generalization of Lemma~\ref{ST}, versions of which appear in both~\cite{BB} and~\cite{Hol}. For completeness we give the proof.

\begin{lemma}\label{crossk}
If $A \subset V = [n]^d$ percolates, and $1 \le \ell \le \dim(V)$, then $W_\ell \ge 1$.
\end{lemma}

\begin{proof}
Consider the following `cubes process'. We start with a set of $m := |A|$ cubes $\C_0 = \{C_1,\ldots,C_m\}$ of dimension zero (i.e., single points), and at each step we choose two of the cubes $C_i,C_j \in \C_t$, such that $d(C_i,C_j) \le 2$, and form $\C_{t+1}$ by replacing them with the cube $[C_i \cup C_j]$. Thus, for each $t \in \N$,
$$\C_{t+1} \; := \; \Big( \C_t \setminus \{C_i,C_j\} \Big) \cup \{[C_i \cup C_j]\}.$$
Since $A$ percolates, we may continue this process until we are left with the single cube, $[n]^d$. Observe that, at each stage of this process, the cubes in $\C_t$ are all disjointly internally spanned by $A$.

Now, simply consider the first time $t$ at which there is a cube $U \in \C_t$ with $\dim(U) \ge \ell$. Then in $\C_{t-1}$ there must exist cubes $S$ and $T$ satisfying conditions $(a)$-$(c)$ above. Hence $W_\ell \ge 1$, as required.
\end{proof}

We shall show that if $\eps \in \RR$ is sufficiently small and $\ell \in \N$ and $p > 0$ are as defined above, then $\Ex_p(W_\ell) \ll 1$. By Lemma~\ref{crossk}, it will follow that $\Pr_p(A\textup{ percolates}) \to 0$ as $d \to \infty$.

In order to bound $\Ex_p(W_\ell)$, we shall need to bound from above the probability that an arbitrary cube $Q$ is internally spanned, given $\dim(Q)$. The following lemma, combined with Theorem~\ref{2^ell}, allows us to do this. Let $P(Q,p)$ denote the probability that $[A] = Q$, given $A \sim \Bin(Q,p)$, for an arbitrary cube $Q$.

\begin{lemma}\label{arbcube}
Let $n,d,\ell \in \N$ and $p > 0$. Suppose $Q \subset [n]^d$ is a cube with $\dim(Q) = \ell$. Then,
$$P(Q,p) \; \le \; P(\ell,p).$$
\end{lemma}

\begin{proof}
We describe a very simple coupling, from which the result follows immediately. Indeed, let $a_1, \ldots, a_d \in \N$, and assume that $a_1 \ge 3$. Let $S \subset [n]^{d+1}$ be an $[a_1] \times \ldots \times [a_d] \times [1]$-cube, and let $T \subset [n]^{d+1}$ be an $[a_1 - 1] \times [a_2] \times \ldots \times [a_d] \times [2]$-cube. Note that $\dim(S) = \dim(T)$. We claim that
$$P(S,p) \; \le \; P(T,p),$$
and moreover that there is an injection $\phi: S \to T$ such that if $A$ percolates in $S$ then $\phi(A)$ percolates in $T$.

Indeed, let $\phi'$ denote the natural isomorphism (as graphs) between $S$ and the subgraph $T'$ of $T$ induced by the vertex set
$$\big\{ \textbf{x} \in T \,:\, x_{d+1} = 1 \big\} \,\cup\, \big\{ \textbf{x} \in T \,:\, x_1 = a_1-1 \textup{ and } x_{d+1} = 2 \big\},$$
and let $\phi$ be the injection (of sets) induced by $\phi'$. Thus, given $A \subset S$, we define $\phi(A) \subset T$ as follows:
\begin{itemize}
\item[$(a)$] If $\textbf{x} \in T$ with $x_{d+1} = 1$, then $\textbf{x} \in \phi(A) \; \Leftrightarrow \; \textbf{x} \in A$,
\item[$(b)$] $(x_1,\ldots,x_d,2) \in \phi(A) \; \Leftrightarrow \; x_1 = a_1-1 \textup{ and } (a_1,x_2\ldots,x_d,1) \in A$.
\end{itemize}
Since the graphs $S$ and $T'$ are isomorphic, it follows immediately that if $[A] = V(S)$, then $V(T') \subset [\phi(A)]$. But $[V(T')] = V(T)$, so $\phi(A)$ percolates in $T$, as required.
\end{proof}

\begin{rmk}
Note that the proof of Lemma~\ref{arbcube} does not work on the torus $\ZZ_n^d$.
\end{rmk}

Now, given $k,m \in \N$, with $k \ge m$, let $\U_\ell(k,m) \subset \U_\ell$ denote the set of pairs $(S,T) \in \U_\ell$ with $\dim(S) = k$ and $\dim(T) = m$, and let $\W_\ell(k,m) = \U_\ell(k,m) \cap \W_\ell$. We write $U_\ell(k,m) = |\U_\ell(k,m)|$ and $W_\ell(k,m) = |\W_\ell(k,m)|$. Thus
$$W_{\ell} \; = \; \sum_{k,m \in \N} W_\ell(k,m).$$
Given $n$, $d$ and $\ell$, as defined above, we shall bound $\Ex_p\big( W_\ell(k,m) \big)$ for each $k,m \in \N$.

The following simple observation will help us bound the number of pairs $(S,T)$ of subcubes of $[n]^d$, with $\dim(S) = k$, $\dim(T) = m$ and $d(S,T) \le 2$.

\begin{obs}\label{ais}
Let $n,d \in \N$ and $a_1, \ldots, a_d \in [n]$, let $k = \sum_{i=1}^d (a_i - 1)$, and suppose that $k \le d$. Then
$$\prod_{i=1}^d a_i(n - a_i + 1) \; \le \; 2^k \, n^d \left( \frac{n-1}{n} \right)^k.$$
\end{obs}

\begin{proof}
The result clearly holds when $a_i \in \{1,2\}$ for each $i$. Suppose that $a_1 \ge 3$; since $k \le d$, it follows that $a_i = 1$ for some $i$, so let $a_2 = 1$. Let $a_1' = a_1 - 1$, $a_2' = 2$ and $a_i' = a_i$ for $i \ge 3$. Then
$$\prod_i a'_i(n - a'_i + 1) \; \ge \; \prod_i a_i(n - a_i + 1)$$
if and only if
$$2(a_1 - 1)(n - 1)(n - a_1 + 2) \; \ge \; a_1n(n - a_1 + 1).$$
But this holds since $a_1 \ge 3$, so $2(a_1-1) \ge a_1$, and $(n - 1)(n - a_1 + 2) \ge n(n - a_1 + 1)$.
\end{proof}

We can now bound $U_\ell(k,m)$, i.e., the number of pairs $(S,T)$, for which the event $(S,T) \in \W_\ell(k,m)$ has positive probability.

\begin{lemma}\label{countU}
Let $n,d,\ell,k,m \in \N$, with $m \le k < \ell \le 2\sqrt{d\log n} < d/4$. Then
$$U_\ell(k,m) \; \le \; 2^{k+m+2} \,d^{m+2} {d \choose k} {d \choose m} \, n^{d+16} \left( \frac{n-1}{n} \right)^k.$$
\end{lemma}

\begin{proof}
Recall that $U_\ell(k,m)$ denotes the number of pairs $(S,T)$ such that $S,T \subset [n]^d$ are cubes, with $\dim(S) = k$ and $\dim(T) = m$, such that $[S \cup T]$ is a subcube with $\dim([S \cup T]) \ge \ell$.

There are at most ${{d+k-1} \choose k} < {{d + k} \choose k}$ ways to choose the dimensions of $S$, since there are at most this many ways of choosing a sequence of $d$ non-negative integers summing to $k$. Similarly there are at most ${{d + m} \choose m}$ ways to choose the dimensions of $T$. Observe that
$${{d + k} \choose k} \; \le \; \left( \frac{d}{d-k} \right)^k {d \choose k} \; \le \; \exp\left( \frac{k^2}{d-k} \right) {d\choose k} \; \le \; n^8 {d \choose k}$$
since $k < \ell \le 2\sqrt{d\log_2 n}$, and $k < d/4$. Similarly ${{d+m} \choose m} \le n^8 {d \choose m}$.

Now, if $S$ is an $[a_1] \times \ldots \times [a_d]$-cube, then there are at most $\prod_i (n - a_i + 1)$ ways to choose its bottom-left corner, and given $S$, there are at most
$$(2d)^{m+2} |S| \; = \; (2d)^{m+2} \prod_i a_i$$
ways to choose the bottom-left corner of $T$. This follows because $[S \cup T]$ is a cube, so the bottom-left corner of $T$ must be within distance $m + 2$ of some point of $S$.

Thus, by Observation~\ref{ais}, there are at most
\begin{eqnarray*}
n^{16} {d \choose k} {d \choose m} (2d)^{m+2}  \prod_i a_i(n - a_i + 1) & \le & 2^{k+m+2} \,d^{m+2} {d \choose k} {d \choose m} \, n^{d+16} \left( \frac{n-1}{n} \right)^k
\end{eqnarray*}
ways of choosing $S$ and $T$, and so the lemma follows.
\end{proof}

Using Theorem~\ref{2^ell}, Lemmas~\ref{arbcube} and~\ref{countU}, and the van den Berg-Kesten Lemma, we can now bound $\Ex_p\big( W_\ell(k,m) \big)$ for all $k,m \in \N$.

\begin{lemma}\label{ems}
Let $n = n(d) \in \N$ satisfy $d \gg \log n$, and let $\alpha > 0$ be a sufficiently large constant. Let $\ell = 2 \big( \sqrt{d \log n} - \alpha \big)$, and let $p = 2^{-s} > 0$ be as defined in~\eqref{sdef}, for some $\eps(d) < 1$. Let $k,m \in \N$, with $m \le k < \ell$ and $k + m + 2 \ge \ell$. Then,
$$\Ex_p\big( W_\ell(k,m) \big) \; \le \; n^{16} \, 2^{\eps \ell \, - \, 2\log d \,+\, O(1)}.$$
Moreover, if $m \ge 6$ then
$$\Ex_p\big( W_\ell(k,m) \big) \; \ll \; 2^{-\ell}.$$
\end{lemma}

\begin{proof}
Recall that if $(S,T) \in \W_\ell(k,m)$ then $S$ and $T$ are disjointly internally spanned. Thus, by the van den Berg-Kesten Lemma, and Lemmas~\ref{arbcube} and~\ref{countU},
\begin{eqnarray*}
\Ex_p\big( W_\ell(k,m) \big) & \le & \sum_{(S,T) \in \U_\ell(k,m)} P(S,p) P(T,p)\\
& \le & 2^{k+m+2} \,d^{m+2} {d \choose k} {d \choose m} \, n^{d+16} \left( \frac{n-1}{n} \right)^k P(k,p) \, P(m,p).
\end{eqnarray*}
Now let $\delta > 0$ be the constant in Theorem~\ref{2^ell}, and recall that $\ell^22^\ell p \le \delta$ if $d$ is sufficiently large, and that $k,m < \ell$. Thus, by Theorem~\ref{2^ell},
$${d \choose k} {d \choose m} \,P(k,p) \, P(m,p) \; \le \; \frac{d^{k+m}}{k!\,m!} \Big( k! \,\lambda^{-k/2} 2^{k^2/4} p^{(k+2)/2} \Big) \Big( m! \,\lambda^{-m/2} 2^{m^2/4} p^{(m+2)/2} \Big).$$
(Here we used the fact that $2^{k^2/4}\, p^{(k+2)/2} \,\gg\, 2^{(k+1)^2/4}\,p^{(k+3)/2}$, and similarly for $m$. This follows because $k,m < \ell$, and $2^\ell p \to 0$ as $d \to \infty$, see also~\eqref{genupper}.) Thus
\begin{equation}\label{eqEx0}
\Ex_p\big( W_\ell(k,m) \big) \; \le \; 2^{k+m+2} \,d^{k+2m+2}\, n^{d+16} \left( \frac{n-1}{n} \right)^k \bigg( \lambda^{-(k+m)/2} \, 2^{(k^2+m^2)/4} p^{(k+m+4)/2} \bigg).
\end{equation}

The rest of the proof is now a straightforward calculation. Indeed, we claim that, given $k \ge m$ with $k + m + 2 \ge \ell$, the right-hand side of~\eqref{eqEx0} is maximized when $k = \ell - 2$ and $m = 0$. To see this, observe first that the function
$$f(k,m) \; := \; 2^{k+m+2} \,d^{k+2m+2}\, n^{d+16} \left( \frac{n-1}{n} \right)^k \lambda^{-(k+m)/2} 2^{(k^2 + m^2)/4} \,p^{(k+m+4)/2}$$
is decreasing in $k$. This follows because
$$\left( \frac{f(k+1,m)}{f(k,m)} \right)^2 \; < \; 4\sqrt{2} \,d^2 \,2^kp \; < \; 1,$$
since $d^2 2^\ell p < 2^{\beta - 2\alpha}$, and $\alpha > 0$ was chosen to be sufficiently large. Thus $f(k,m)$ is maximized with $k = \ell - m - 2$.

Now suppose that $k = \ell - m - 2$ and $m > 0$. Then either $0 < m < \ell/4$, in which case
$$\left( \frac{f(k+1,m-1)}{f(k,m)} \right)^2 \; > \; \frac{1}{d^2} \left(\frac{n-1}{n} \right)^2 2^{k - m + 1} \; \gg \; 1,$$
since $k - m > \ell/2 - 2$, or $m \ge \ell/4$, in which case
$$\frac{f(k+\ell/4,m-\ell/4)}{f(k,m)} \; \ge \; \left( \frac{1}{2d} \right)^{\ell/4} \,2^{\ell^2/32} \; \gg \; 1.$$
So $f(k,m)$ is maximized when $k = \ell - 2$ and $m = 0$, as claimed, and hence
$$\Ex_p\big( W_\ell(k,m) \big) \; \le \; 2^\ell \,d^\ell\, n^{d+16} \left( \frac{n-1}{n} \right)^{\ell-2} \lambda^{-\ell/2+1} \, 2^{(\ell-2)^2/4} p^{(\ell+2)/2}.$$

Now, recall that
$$d\log n \; = \; \frac{\big( \ell + 2\alpha \big)^2}{4}  \; = \; \frac{\ell^2}{4} \,+\, \alpha \ell \,+\, O( 1 )$$
and note that
\begin{eqnarray*}
s(\ell + 2) & = & \Big( 2\sqrt{d \log n} + 2 \log d - \beta \Big) \Big( \ell + 2 \Big) \; = \; \Big( \ell + 2\alpha + 2 \log d - \beta \Big) \Big( \ell + 2 \Big)\\
& = & \ell^2 \,+\, \ell \Big( 2\log d + 2 + 2\alpha - \beta \Big) \, + \, 4\log d \, + \, O(1),
\end{eqnarray*}
so
$$d \log n \,+\, \frac{\ell^2}{4} \,-\, \frac{s(\ell+2)}{2} \; = \; - \ell \big( \log d  + 1 - \beta/2 \big) \, - \, 2\log d \, + \, O(1).$$
Hence,
\begin{eqnarray*}
\Ex_p\big( W_\ell(k,m) \big) & \le & 2^\ell \,d^\ell\, n^{d+16} \left( \frac{n-1}{n} \right)^{\ell-2} \lambda^{-\ell/2+1} \, 2^{(\ell-2)^2/4} p^{(\ell+2)/2}\\
& \le & 8 \lambda \, n^{16} \left( \frac{d(n-1)}{n\sqrt{\lambda}} \right)^\ell n^d \, 2^{\ell^2/4 - s(\ell+2)/2} \\
& \le & n^{16} \left( \frac{d(n-1)}{n\sqrt{\lambda}} \right)^\ell 2^{- \ell ( \log d  + 1 - \beta/2) \, - \, 2\log d \,+\, O(1)}.
\end{eqnarray*}
But
$$\left( \frac{n-1}{n} \right) \frac{2^{\beta/2 - 1}}{\sqrt{\lambda}}  \; = \; \left( \frac{n-1}{n} \right) \frac{1}{2\sqrt{\lambda}} \left( \sqrt{4\lambda} \left( \frac{n}{n-1} \right) 2^\eps \right) \; = \; 2^\eps.$$
Thus
\begin{eqnarray*}
\Ex_p\big( W_\ell(k,m) \big) & \le & n^{16} \, 2^{\eps \ell \, - \, 2\log d \,+\, O(1)},
\end{eqnarray*}
as required.

To show that if $m \ge 6$ then $\Ex_p\big( W_\ell(k,m) \big) \ll 2^{-\ell}$, note that, by the argument above, if $m \ge 6$ then $f(k,m)$ is maximized when $k = \ell - 8$ and $m = 6$. Thus
$$\Ex_p\big( W_\ell(k,m) \big) \; \le \; 2^{\ell+ O(1)} \,d^{\ell+6}\, n^{d+16} \left( \frac{n-1}{n} \right)^{\ell} \lambda^{-\ell/2} \, 2^{\ell^2/4 - 4\ell} p^{(\ell+2)/2} \; \le \; d^4 n^{16} \, 2^{\eps \ell - 3 \ell + O(1)},$$
by the same calculation as before. Since $\ell \gg \log d + \log n$ and $\eps < 1$, the result follows.
\end{proof}

The lower bound in Theorems~\ref{hypercube} and~\ref{n^d} now follow easily.

\begin{thm}\label{lower}
Let $n = n(d)$ be a function such that $d \gg \log n \ge 1$. Let
$$p \; = \; 4\lambda \left( \frac{n}{n-1} \right)^2 \frac{1}{d^2}\left( 1 \,+\, \frac{\log d - 16\log n + 16}{\sqrt{d \log n}} \right) 2^{- 2 \sqrt{d\log n}},$$
and let $A \sim \Bin([n]^d,p)$. Then
$$\Pr_p\big([A] = [n]^d \big) \to 0$$
as $d \to \infty$.
\end{thm}

\begin{proof}
Let
$$\eps \; = \; \eps(d) \; := \; \frac{1}{2} \,\log\left[ 1 \,+ \, \frac{\log d - 16\log n + 16}{\sqrt{d \log n}} \right] \; < \; 1,$$
and let
$$\beta = \log \left[ 4\lambda \left( \frac{n}{n-1} \right)^2 \right] \,+\, 2\eps.$$
Set $s = 2\sqrt{d \log n} + 2 \log d - \beta$,  and note that $p = 2^{-s}$. Let $\alpha > 0$ be a sufficiently large constant, and set $\ell = 2 \big( \sqrt{d \log n} - \alpha \big)$.

By Lemma~\ref{crossk}, we have
$$\Pr_p\big([A] = [n]^d \big)  \; \le \; \Pr_p\big( W_\ell \ge 1 \big) \; \le \; \Ex_p\big(W_\ell \big),$$
and by Lemma~\ref{ems}, if $d$ is sufficiently large,
$$\Ex_p(W_\ell) \; \le \; \sum_{\substack{0 \,\le\, m \, \le \, k \, < \, \ell \\[+0.5ex] k + m \,\ge\, \ell - 2}} \Ex_p \big( W_\ell(k,m) \big) \; \le \; n^{16} 2^{\eps\ell - 2\log d + O(1)} \,+\, {\ell \choose 2} 2^{-\ell},$$
since $1 \le \ell - k \le m + 2$, so we have a bounded number of choices for $k$ and $m$ with $m \le 5$.
But $\log(1 +  z) \le z$, so
$$\eps \ell \,-\, 2\log d \,+\, 16\log n \; \le \; \Big( \log d - 16\log n + 16 \Big) \,-\, 2\log d \,+\, 16\log n \; = \; - \log d + 16.$$
Thus,
$$\Pr_p\big( A\textup{ percolates} \big) \; \le \; O\left( \frac{1}{d} \right) \,+\, \big( 2 d \log n \big) 2^{-\ell} \; \to \; 0$$
as $d \to \infty$, as required.
\end{proof}

\section{Open problems}\label{probsec}

The work in this paper is only one step towards the main aim of this line of research: to determine $p_c([n]^d,r)$ for arbitrary functions $n = n(t)$, $d = d(t)$ and $r = r(t)$, with $n + d \to \infty$ as $t \to \infty$. In this section we shall describe some of the most natural next steps which one could take, and state some of the most attractive open problems.

In Theorem~\ref{n^d} we determined the sharp threshold for $p_c([n]^d,2)$ for all functions $n = n(d)$ with $d \gg \log n$. On the other hand, Balogh, Bollob\'as, Duminil-Copin and Morris~\cite{alldr} determined a sharp threshold when $d$ is constant. To be precise, they proved that
$$p_c([n]^d,2) \; = \; \left( \frac{\lambda(d,2) + o(1)}{\log n} \right)^{d-1},$$
for an explicit constant $\lambda(d,2) > 0$ satisfying $\lambda(d,2) = \frac{d-1}{2} + o(1)$ as $d \to \infty$. (In fact they proved a more general result, determining a sharp threshold for all $2 \le r \le d$.)

\begin{prob}\label{Olog}
Determine $p_c([n]^d,2)$ for $1 \ll d = O(\log n)$.
\end{prob}

When $d \le \log n$ the critical droplet is no longer a hypercube, since $[2]^d$ is not likely to grow. Thus the first step in solving Problem~\ref{Olog} is likely to be the following problem.

Let $P(k,\ell,p)$ denote the probability that $A \sim \Bin([k]^\ell,p)$ percolates.

\begin{prob}
Determine $P(k,\ell,p)$ when $k^\ell p$ is sufficiently small.
\end{prob}

An even more important extension of Theorems~\ref{hypercube} and~\ref{n^d} would be to study the $r$-neighbour process, for a general (but fixed) value of $r$. The problem seems to become much harder when $r \ge 3$; in fact, we are not yet able to prove even the following fairly weak conjecture, even in the case $r = 3$.

\begin{conj}\label{rconj}
For $r$ fixed,
$$p_c([2]^d,r) \; = \; \exp\Big( - \Theta\left( d^{1 / 2^{r-1}} \right) \Big).$$
\end{conj}

A related question, which is often useful in studying the critical threshold, asks for the extremal percolating sets. To be precise, given a graph $G$ and an integer $r$, let
$$m(G,r) \; := \; \min\big\{ |A| \,:\, A \textup{ percolates in $r$-neighbour bootstrap on $G$} \big\}.$$
Note that, by Lemma~\ref{minl}, we have
$$m([n]^d,2) \; = \; \left\lceil \frac{d(n-1)}{2} \right\rceil + 1,$$
since it is straightforward to construct a percolating set with this many elements. When $d$ and $r$ are fixed, Pete~\cite{Pete} gave various bounds on $m([n]^d,r)$; in particular, he showed that $m([n]^d,d) = n^{d-1}$.

For the $r$-neighbour process on the hypercube, with $r$ fixed, we have the following construction.

\begin{prop}\label{construc}
Let $r \in \N$. Then
$$m([2]^d,r) \; \le \; \frac{1 + o(1)}{r}{d \choose {r-1}}$$
as $d \to \infty$.
\end{prop}

\begin{proof}
Let $d \ge 2r$, and let $A_1$ be a system of $r$-subsets of $[d]$ such that every $(r-1)$-set is contained in at least one set of $A_1$. By a theorem of R\"odl~\cite{Rodl}, which resolved a long-standing conjecture of Erd\H{o}s and Hanani~\cite{EH}, there exists such a set $A_1$ with
$$|A_1| \; = \; \frac{1 + o(1)}{r} {d \choose {r-1}}$$
as $d \to \infty$. Let $A_2$ denote all the $(r-2)$-subsets of $[d]$. We claim that $A := A_1 \cup A_2$ percolates in $r$-neighbour bootstrap percolation on $[2]^d$.
To see this, simply note that every $(r-1)$-subset of $[d]$ has $r-1$ neighbours in $A_2$, and at least one neighbour in $A_1$. Hence all the $(r-1)$-sets are in $[A]$, and the $(r-1)$-sets percolate in $[2]^d$, since $d \ge 2r$.
\end{proof}

We conjecture that the construction in Proposition~\ref{construc} is close to being extremal.

\begin{conj}
Let $r \in \N$. Then
$$m([2]^d,r) \; = \; \Theta\big( d^{r-1} \big).$$
\end{conj}

It might even be true that Proposition~\ref{construc} is essentially sharp, i.e., that
$$m([2]^d,r) \; = \; \big( 1 + o(1) \big) \frac{d^{r-1}}{r!}.$$
Since $m([2]^d,2) = \lceil d/2 \rceil + 1$, it follows easily that $m([2]^d,r) \ge 2^{r-3}(d - r + 4)$, since any set which $r$-percolates in $Q \cong [2]^d$ must $(r-1)$-percolate in each half of $Q$ separately. Amazingly, however, we know of no non-trivial lower bound, even for $m([2]^d,3)$.

Returning to the case $r = 2$, we should also like to know the value of $p_c(\ZZ_n^d,2)$, the critical probability for percolation on the $d$-dimensional torus, when $d \gg \log n$. As we noted in Remark~\ref{torus}, when $n \ge 4$ this would follow by the proof of Theorem~\ref{n^d}, together with a modified version of Lemma~\ref{arbcube}. However, when $n = 3$ the critical probability will be completely different, since we can span $\ZZ_3^{2\ell}$ with $\ell + 1$ infected sites, and there are at least $3^{\ell^2}$ such minimal spanning sets. Nevertheless, it is likely that many of the ideas and techniques in this paper will be useful in attacking the following problem.

\begin{prob}
Determine $p_c\big(\ZZ_3^d,2\big)$.
\end{prob}

Finally, on the hypercube with $r = 2$, we ask for an even sharper result than Theorem~\ref{hypercube}.

\begin{prob}
Determine $\alpha \in [1,2]$, if it exists, such that, for any $\eps > 0$,
$$\ds\frac{16\lambda}{d^2} \left( 1 \,+\, \frac{(\log d)^{\alpha-\eps}}{\sqrt{d}} \right) \, 2^{-2\sqrt{d}} \; \le \; p_c([2]^d,2) \; \le \; \ds\frac{16\lambda}{d^2}  \left(1 \,+\, \frac{(\log d)^{\alpha+\eps}}{\sqrt{d}} \right) \, 2^{-2\sqrt{d}}$$
if $d$ is sufficiently large.
\end{prob}

\section{Appendix A: Proof of the technical lemma}

In this section we shall prove Lemma~\ref{tech} for an arbitrary sequence $(a_1,a_2,\ldots)$ of positive real numbers, and an arbitrary $\lambda$, satisfying the following conditions:\\
$$(a) \; 1 < \lambda < 2, \;\; (b) \; a_1 = \lambda, \; \; (c) \; \ds\sum_{m=1}^\infty (-1)^{m+1} a_m = 1, \;\; (d) \; \ds\frac{a_m}{a_{m+1}} \ge \frac{6}{\lambda} \;\textup{ for every $m \in \N$.}$$
Recall that
$$a_m = \ds\frac{(2\lambda)^m}{2^{m^2}m!}$$
for each $m \in \N$, and observe that these conditions hold for the sequence $(a_m)$.

The proof of Lemma~\ref{tech} is elementary, and based on a series of straightforward induction arguments. We shall use the following simple properties of $a_m$, which follow from the four conditions stated above.

\begin{lemma}\label{aprops}
Let $a_m$ be as described above. Then
\begin{enumerate}
\item[$(a)$] $\ds\frac{a_m}{a_{m+1}} = \ds\frac{2^{2m}(m+1)}{\lambda} \ge \frac{6}{\lambda}$,\\
\item[$(b)$] $a_m > \ds\sum_{i=m+1}^\infty a_i$ for every $m \in \N$,\\
\item[$(c)$] $\ds\sum_{i=1}^{2m} (-1)^{i+1} a_i \; < \; 1 \; < \; \sum_{i=1}^{2m-1} (-1)^{i+1} a_i$ for every $m \in \N$, \\
\item[$(d)$] $\ds\frac{1}{2} \, a_{m+1} \; < \; \left| \left( \ds\sum_{i=1}^m (-1)^{i+1} a_i \right) - 1 \right| \; < \; a_{m+1}$,\\
\item[$(e)$] $\left| \left( \ds\sum_{i=1}^m (-1)^{i+1} a_i \right) - 1 \right| \; \ge \; \ds\frac{3}{\lambda} \,\left|\left( \ds\sum_{i=1}^{m+1} (-1)^{i+1} a_i \right) - 1 \right|$.
\end{enumerate}
\end{lemma}

\begin{proof}
Part $(a)$ is trivial, and part $(b)$ follows easily from $(a)$, since $6/\lambda > 2$. For $(c)$, recall that $\sum_{m=1}^\infty (-1)^{m+1} a_m = 1$, and suppose that  $\sum_{i=1}^{2m} (-1)^{i+1} a_i \ge 1$ for some $m \in \N$. Then by $(b)$,
$$\sum_{i=1}^t (-1)^{i+1} a_i \; \ge \; 1 + a_{2m+1} - \ds\sum_{i=2m+2}^\infty a_i \; > \; 1$$ for every $t \ge 2m + 1$, which is a contradiction. The proof of the other bound is identical.

Now, $(c)$ implies immediately that $$a_{m+1} - a_{m+2} \; < \; \left| \left( \ds\sum_{i=1}^m (-1)^{i+1} a_i \right) - 1 \right| \; < \; a_{m+1},$$ and by $(a)$ we have $2a_{m+2} < a_{m+1}$, so $(d)$ follows. Finally, by $(a)$ and $(d)$,
\begin{eqnarray*}
\left| \left( \sum_{i=1}^m (-1)^{i+1} a_i \right) - 1 \right| & \ge & \frac{1}{2} \, a_{m+1} \; \ge \; \ds\frac{3}{\lambda} \, a_{m+2}\; \ge \; \ds\frac{3}{\lambda} \, \left|\left( \ds\sum_{i=1}^{m+1} (-1)^{i+1} a_i \right) - 1 \right|,
\end{eqnarray*}
and so $(e)$ holds.
\end{proof}

We shall also use the following lemma.

\begin{lemma}\label{heffect}
Let $2 \le \ell \in \N$ and $1 < \lambda < 2$, and let $g(m) \ge 0$ for each $1 \le m \le \ell$. Suppose that $f(1) \le \lambda/2$, and that $f(t)$ satisfies the recursion
$$f(t) \: \le \: 1 \,-\, \frac{\lambda}{2} \, + \, (\lambda - 1)\big( 1 + g(t-1) \big) f(t-1) \, + \, \sum_{m=1}^{t-1} f(m)g(m),$$
for every $2 \le t \le \ell$. Then
$$f(\ell) \; \le \; \frac{\lambda}{2} \, \exp\left( \frac{1}{2- \lambda} \ds\sum_{m=1}^{\ell-1} g(m) \right).$$
\end{lemma}

\begin{proof}
For each $\ell \in \N$, let
$$H(\ell) = \frac{\lambda}{2} \exp\left( \frac{1}{2- \lambda} \ds\sum_{m=1}^{\ell-1} g(m) \right).$$
We shall prove by induction on $\ell$ that $f(\ell) \le H(\ell)$. The case $\ell = 1$ follows since $f(1) \le \lambda/2$. Note also that $$\big( 1 + g(m) \big) H(m) \; \le \; e^{g(m)}H(m) \; \le \; H(m+1),$$
since $\ds\frac{1}{2 - \lambda} \ge 1$ and $g(m) \ge 0$.

Let us think of $g(m)$ and $H(m)$ as step functions defined on $\RR^+$, i.e., define $g(z) = g(\lfloor z \rfloor)$, so that $\ds\int_1^\ell g(z) \,dz = \ds\sum_{m=1}^{\ell-1} g(m)$, and similarly for $H(m)$, so
$$H(z) \; = \; b \exp\left( c \int_1^{\lfloor z \rfloor} g(y)\, dy \right) \; \le \; b \exp\left( c \int_1^z g(y) \,dy \right),$$
where $b = \ds\frac{\lambda}{2}$ and $c = \ds\frac{1}{2- \lambda}$. We shall use the fact that
\begin{eqnarray*}
\sum_{m=1}^{\ell-1} g(m)H(m) & \le & b \int_1^\ell g(z) \exp\left( c \int_1^z g(y) \,dy \right) \,dz \\
& = & \left[ \frac{b}{c}\, \exp\left( c \int_1^z g(y) \,dy \right) \right]_1^\ell \; = \; \frac{b}{c} \, \exp\left(c \int_1^\ell g(y) \,dy \right) \: - \: \frac{b}{c},
\end{eqnarray*}
which follows by the fundamental theorem of calculus. By the induction hypothesis we have $f(m) \le H(m)$ for every $1 \le m \le \ell - 1$. Thus
\begin{eqnarray*}
f(\ell) & \le & 1 \,-\, \ds\frac{\lambda}{2} \, + \, (\lambda - 1)\big( 1 + g(\ell-1) \big) H(\ell-1) \, + \, \sum_{m=1}^{\ell-1} g(m)H(m)\\
& \le & 1 \,-\, \frac{\lambda}{2} \, + \, (\lambda - 1)H(\ell)  \, + \, \frac{b}{c} \, \exp\left(c \int_1^\ell g(y) \,dy \right) \, - \, \frac{b}{c}\\
& = & H(\ell) \left(\lambda \, - \, 1  \, + \, \frac{1}{c} \right) \, + \, \left( 1 - \ds\frac{\lambda}{2} \right)(1 - \lambda)  \; < \; H(\ell)
\end{eqnarray*}
as required, since $\ds\frac{1}{c} = 2 - \lambda$, $\ds\frac{b}{c} = \lambda \left( 1 - \ds\frac{\lambda}{2} \right)$, and $1 < \lambda < 2$.
\end{proof}

We are now ready to prove Lemma~\ref{tech}.

\begin{proof}[Proof of Lemma~\ref{tech}]
Let $f$, $g$ and $h$ be as in the lemma. First we shall show that $f(t)$ does not `jump'.\\[-1ex]

\noindent \ul{Claim 1}: For every $t \in \N$,
$$0 \; < \; \frac{\lambda}{2} \, h(t-1)f(t-1) \; \le \; f(t) \; \le \; \frac{3\lambda}{2} \, h(t-1)f(t-1).$$

\begin{proof}[Proof of Claim 1]
We shall prove the claim by induction on $t$. For $t = 1$ we have $f(1) = \lambda/2$, so the statement holds since $h(0) = f(0) = 1$. So let $t \ge 2$ and assume that the claim is true for $1, \ldots, t-1$. The induction step follows from the following two easy subclaims.\\[-1ex]

\noindent \ul{Subclaim 1}: For every $1 \le m \le t-1$, $$a_m h(t-m) f(t-m) \; \ge \; 3a_{m+1} h(t-m-1)  f(t-m-1).$$

\begin{proof}[Proof of Subclaim 1]
This follows from the induction hypothesis, and Lemma~\ref{aprops}$(a)$. To be precise, by the induction hypothesis we have
$$f(t-m) \; \ge \; \frac{\lambda}{2} \, h(t-m-1)f(t-m-1),$$
since $1 \le t-m \le t - 1$. Thus, by Lemma~\ref{aprops}$(a)$,
$$\frac{a_mf(t-m)}{a_{m+1}h(t-m-1)f(t-m-1)} \;  \ge \; \frac{a_m}{a_{m+1}} \cdot \frac{\lambda}{2} \; \ge \; 3.$$
Since $h(t-m) \ge 1$, the subclaim follows.
\end{proof}

\noindent \ul{Subclaim 2}: For every $1 \le m \le t-1$, $$a_m h(t-m) f(t-m) \; \ge \; 2 \, \sum_{i = m+1}^t a_i h(t-i)  f(t-i).$$

\begin{proof}[Proof of Subclaim 2]
We prove the subclaim by induction on $t - m$. When $t - m = 1$, then by Lemma~\ref{aprops}$(a)$,
$$a_{t-1} h(1)f(1) \; \ge \; \frac{\lambda}{2} \cdot \left( \frac{6}{\lambda} \right) a_t \; \ge \;  2a_t h(0)f(0),$$
as required, since $f(0) = h(0) = 1$, $f(1) = \lambda/2$ and $h(1) \ge 1$.

So let $t - m \ge 2$, and assume the subclaim holds for $t - m - 1$. Then,
$$a_{m+1} h(t-m-1) f(t-m-1) \; \ge \; 2 \sum_{i = m+2}^t a_i h(t-i)  f(t-i),$$
and so, by Subclaim 1,
\begin{eqnarray*}
a_m h(t-m) f(t-m) & \ge & 3a_{m+1} h(t-m-1)  f(t-m-1)\\
& \ge & 2 \, \sum_{i = m+1}^t a_i h(t-i)  f(t-i),
\end{eqnarray*}
as required.
\end{proof}

Now, by the induction hypothesis, $f(t-i) > 0$ for all $i \in [t-1]$. Recall that
$$f(t) = \ds\sum_{m=1}^t (-1)^{m+1} a_m h(t-m) f(t-m),$$
and that $a_1 = \lambda$. Thus, by Subclaim~2, we have
\begin{eqnarray*}
| f(t) - \lambda h(t-1)f(t-1)| & \le & \sum_{m=2}^t a_mh(t-m)f(t-m)\; \le \; \frac{\lambda}{2} \, h(t-1)f(t-1).
\end{eqnarray*}
The induction step, and hence the claim, follows immediately.
\end{proof}

Lemma~\ref{tech} will follow from Claim~1 once we have re-written the recursion formula for $f(t)$ in the following, more useful way.\\[-1ex]

\noindent \ul{Claim 2}: For every $1 \le t \le \ell$,
\begin{eqnarray*}
f(t) & = & 1 \:-\: \frac{\lambda}{2} \:+\: \ds\sum_{m=1}^t \left( \left( \sum _{i=1}^m (-1)^{i+1} a_i \right) - 1 \right) h(t-m)f(t-m)\\
&& \hspace{6cm} + \; \ds\sum_{m=1}^t \Big( \big(h(t-m) - 1 \big)f(t-m) \Big).
\end{eqnarray*}

\begin{proof}[Proof of Claim~2]
Consider the sum $f(t) + f(t-1) + \ldots + f(0)$, and recall that
$$f(m) \; = \; \ds\sum_{i=1}^m (-1)^{i+1} a_i h(m-i) f(m-i)$$
for every $m \in [t]$. The claim follows by substituting this sum for each $f(m)$, and then grouping terms according to which $f(m)$ they contain. To spell it out,
\begin{eqnarray*}
f(t) + \ldots + f(0) & = & f(0) \,+\, f(1) \: + \: \sum_{m=2}^t \, \sum_{i=1}^m (-1)^{i+1} a_i h(m-i) f(m-i)\\
& = & f(0) \,+\, f(1) \,-\, a_1 h(0) f(0) \, + \, \sum_{k=0}^{t-1} \left( \sum_{i=1}^{t-k} (-1)^{i+1} a_i \right) h(k) f(k)\\
& = & 1 \:-\: \frac{\lambda}{2} \: + \: \sum_{m=1}^t \left( \sum_{i=1}^m (-1)^{i+1} a_i \right) h(t-m) f(t-m),
\end{eqnarray*}
Subtracting $f(t-1) + \ldots + f(0)$, we obtain the claimed recursion.
\end{proof}

We are now ready to prove the claimed lower bound on $f(\ell)$. To keep the formulae concise, for each $m \in \N$ define $b(m) = \sum_{i=1}^m (-1)^{i+1} a_i - 1$. Recall that by Lemma~\ref{aprops}$(c)$ and $(e)$, we have $b(m) > 0$ if $m$ is odd, and
$$\left| b(m)\right| \; \ge \; \frac{3}{\lambda} \left|b(m+1) \right|.$$
Furthermore, $f(t-m) \ge \ds\frac{\lambda}{2} \, h(t-m-1)f(t-m-1)$, by Claim~1, and $h(t-m) \ge 1$. Thus, if $m$ is odd,
\begin{align*}
& b(m) h(t-m)f(t-m) \: + \:  b(m+1) h(t-m-1)f(t-m-1)\\
& \hspace{2.5cm} \ge \; \left( \frac{|b(m)|}{|b(m+1)|} \cdot \frac{\lambda}{2} \,-\, 1 \right) \big| b(m+1) \big| \,h(t-m-1)\,f(t-m-1) \; \ge \; 0,
\end{align*}
Therefore
$$\sum_{m=1}^t b(m)h(t-m)f(t-m) \; \ge \; 0,$$
since each pair of terms is positive, and if $t$ is odd then the final term is also positive. Recall that $h(t) \ge 1$ and $f(t) \ge 0$ for every $t \in \N$ (by Claim~1), and note that therefore
$$\sum_{m=1}^{t-1} \big( h(t-m) - 1 \big) f(t-m) \; \ge \; 0.$$
Thus, by Claim~2, we have
$$f(t) \; \ge \; 1 \,-\, \frac{\lambda}{2}$$
for every $1 \le t \le \ell$.

The upper bound follows from Lemma~\ref{heffect} and the following claim.\\[-1ex]

\noindent \ul{Claim 3}: Let $t \ge 2$. Then
$$f(t) \: \le \: 1 \,-\, \frac{\lambda}{2} \, + \, b(1)h(t-1) f(t-1) \, + \, \sum_{m=0}^{t-1} f(m)g(m).$$

\begin{proof}[Proof of Claim~3]
Recall that $h(t-m) - 1 \le g(t-m)$ for every $m \in [t-1]$, so the claim will follow from Claim~2 if
$$\ds\sum_{m=1}^t b(m)h(t-m)f(t-m) \; \le \; b(1)h(t-1) f(t-1).$$
But this follows as before. Indeed, if $m$ is even, then
$$b(m) h(t-m)f(t-m) \: + \: b(m+1) h(t-m-1)f(t-m-1) \; \le \; 0$$
by the argument above, and if $t$ is even then the last term is negative.
\end{proof}

Finally, recall that by Lemma~\ref{heffect}, if $1 < \lambda < 2$, $f(1) \le \lambda/2$, and $f(t)$ satisfies the recursion
$$f(t) \: \le \: 1 \,-\, \frac{\lambda}{2} \, + \, (\lambda - 1)\big( 1 + g(t-1) \big) f(t-1) \, + \, \sum_{m=1}^{t-1} f(m)g(m),$$
for every $2 \le t \le \ell$, then we have our desired upper bound on $f(\ell)$. But $g(0) = 0$, $b(1) = \lambda - 1$, and $h(t-1) \le 1 + g(t-1)$. Thus, by Claim~3, the conditions of Lemma~\ref{heffect} hold, and the lemma follows.
\end{proof}

\section{Acknowledgements}

The authors would like to thank the Institute for Mathematical Sciences at the National University of Singapore for their hospitality during May and June of 2006, when this research was begun. The authors would also like to thank the anonymous referee for an extremely careful reading of the proof, and for several very useful suggestions which shortened and simplified parts of Sections~\ref{2^dsec} and~\ref{lowersec}. Finally, they would like to thank Wojciech Samotij for writing the computer programs used in Lemmas~\ref{Yevensmall} and~\ref{Yoddsmall}.

\end{document}